%% file: dirichlet_control.tex
\documentclass[11pt, twoside]{article}

\usepackage{amsmath,amsthm}

\usepackage[utf8]{inputenc}

\usepackage{enumerate}

\usepackage{fancyhdr}
\usepackage{cite}
\usepackage{enumitem}

\usepackage{graphics}

\usepackage{algcompatible}
\usepackage{algorithm}

\usepackage[demo]{graphicx}
\usepackage{caption}
\usepackage{subcaption}

\usepackage{xargs}                      
\usepackage[colorinlistoftodos,prependcaption,textsize=tiny]{todonotes}
\usepackage[normalem]{ulem}

\newcommandx{\pcomment}[2][1=]{\todo[linecolor=red,backgroundcolor=red!25,bordercolor=red,#1]{#2}}
\newcommandx{\kcomment}[2][1=]{\todo[linecolor=blue,backgroundcolor=blue!25,bordercolor=blue,#1]{#2}}

\newif\ifpreprint

\usepackage{hyperref}
\hypersetup{%
    colorlinks=True, 
    citecolor=blue,
    linkcolor=black}

\usepackage[charter]{mathdesign}

\newcommand{\rom}[1]{\uppercase\expandafter{\romannumeral #1\relax}}

\usepackage{environ}
\NewEnviron{cases2}[1]{%
                      \begin{equation}\label{#1}\left\{ \begin{alignedat}{2} \BODY \end{alignedat}\right. \end{equation} }
\NewEnviron{cases3}[1]{%
                      \begin{equation}\label{#1} \begin{alignedat}{2} \BODY \end{alignedat} \end{equation} }
\NewEnviron{cases22}{%
                      \begin{equation}\left\{ \begin{alignedat}{2} \BODY \end{alignedat}\right. \end{equation} }

  \theoremstyle{definition}
  \newtheorem{theorem}{Theorem}[section]

  \newtheorem{lemma}[theorem]{Lemma}
  \newtheorem{definition}[theorem]{Definition}
  \newtheorem{remark}[theorem]{Remark}
  \newtheorem{remarks}[theorem]{Remarks} 
  
  \newtheorem*{problem}{Dirichlet control problem}

  \newtheorem*{assumption*}{Assumption}

  \numberwithin{equation}{section}
    
\input{slidesymbols}

\setenumerate[0]{label=(\alph*)}

\newcommand{\eps}{{\varepsilon}}

\newcommand{\ben}{\begin{equation}}
\newcommand{\een}{\end{equation}}

\newcommand{\benn}{\begin{equation*}}
\newcommand{\eenn}{\end{equation*}}

\usepackage[colorinlistoftodos]{todonotes}

\usepackage{authblk}
\usepackage[T1]{fontenc}
\usepackage[utf8]{inputenc}

\newcommand{\reff}{\text{ref}}

\title{A multimaterial topology optimisation approach to Dirichlet control with piecewise constant functions }
\author{Kevin Sturm}

\begin{document}
\maketitle
\tableofcontents

\begin{abstract}
    In this paper we study a Dirichlet control problem for the Poisson equation, where the control is assumed to be piecewise constant function which is allowed to take $M\ge 2$ different values. The space of admissible Dirichlet controls is non-convex and therefore standard derivatives in Banach spaces are not applicable. Furthermore piecewise constant functions do not belong to $H^{\frac12}$ standard extension techniques to consider the weak solution of the Dirichlet problem do not apply. Therefore we resort to the notion of very weak solutions of the state equation in $L^p$ spaces. We then study the differentiability of the shape-to-state operator of this problem and derive the first order necessary optimality conditions using the topological state derivative. In fact we prove the existence of the weak topological state derivative introduced in \cite{a_BAMAST_2023a} for a multimaterial shape functional which is then expressed via an adjoint variable. The topological derivative resembles formulas found for derivative in the more standard Dirichlet control problems. 
   In the final part of the paper we show how to apply a multimaterial level-set algorithm with the finite element software NGSolve \cite{Schoeberl2014} and present several numerical examples in dimension three. 

\end{abstract}

\section{Introduction}
In this paper we study a Dirichlet control problem where the Dirichlet control is assumed to be a piecewise constant function with prescribed bounds. We consider the Dirichlet control problem as a topology optimisation as follows. Let $\Omega\subset \VR^d$ be a bounded domain with smooth boundary $\partial \Omega$ and $a,b\in \VR^M$  be two vectors and $M\ge2$ an integer. In this paper we study the minimisation of the value-function:
\begin{equation}
    \Cj(S_1,\ldots, S_M):= \min_{a\le \alpha \le b} \int_\Omega (u[S_1,\ldots,S_M,\alpha]-u_{\reff})^2\;dx + \lambda |\alpha|^2, \quad \lambda >0
\end{equation}
over disjoint and open sets $S_1,\ldots, S_M\subset \partial\Omega$ with $\cup_{k=1}^M\overline{S}_k=\partial \Omega$. The function $u=u[S_1,\ldots,S_M,\alpha]$ solves
\begin{cases2}{dirichlet_intro}
    -\Delta u && = f & \quad \text{ in } \Omega, \\
    u_{|S_i} && = \alpha_i & \quad \text{ for } i=1,\ldots, M,
\end{cases2}
 The notation $a\le x\le b$, $x\in \VR^M$ indicates component wise inequalities, that is, $a_i\le x_i\le b_i$ for $i\in \{1,\ldots, M\}$ and $|\alpha|$ denotes the Euclidean norm of $\alpha=(\alpha_1,\ldots,\alpha_M)^\top \in \VR^M$. Here $u_{\reff}\in L^2(\Omega)$ denotes a target function which we try to reach by optimising over the shapes $S_1,\ldots,S_M$. In other words the control the Dirichlet boundary conditions and choose controls of piecewise constant functions taking $\le M$ values $\alpha_1,\ldots,\alpha_M$, respectively. If a set $S_i=\emptyset$ is empty, then the control $\alpha_i$ is inactive.  

There are several challenges associated with this optimisation problem. 
\begin{itemize}
    \item We need to insert a hole in the boundary portion $\partial \Omega$.
    \item The topology perturbation of $\alpha_S$ leads to a Dirac distribution on the boundary and thus to a solution to the  associated Dirichlet problem with low regularity.
\end{itemize}
We will extend the framework of \cite{a_BAMAST_2023a} to the situation of boundary perturbations allowing us to treat the Dirichlet control problem described above.

Dirichlet control problems have been extensively studied in optimal control along with their numerical analysis; see for instance \cite{a_MARAVE_2013a,a_MARAVE_2008a,a_KUVE_2007a}. In \cite{a_BE_2004a} very weak solutions to the Dirichlet problem have been examined numerically and it was observed that for linear finite elements one can actually simply interpolate $L_2$ Dirichlet data to continuous finite elements, which is equivalent to solving a very week solution. This has a computational advantage as the computation of the very weak solution amounts to solving additional elliptic problems to create new test functions. This is avoided by interpolation of the Dirichlet data. We will make use of that fact in the numerical simulations. 

Regarding topology optimisation there is also a vast theoretical foundation for a variety of problems. We refer the reader the monographs \cite{b_NOSO_2013a,b_NOSOZO_2019a,b_NOSO_2020a} and the introductory paper \cite{a_AM_2021a}. While often topological perturbations inside of the domain of definition of the partial differential equation are considered only few papers deal with perturbations on the boundary. An example where a boundary perturbation is considered is for instance \cite{a_LANASO_2011a}. For the theoretical treatment of topological perturbations in lower dimensions we refer to \cite{a_DE_2017a}. Despite these papers there are only few papers dealing with topological boundary perturbations. 

This paper will utilize the topological derivative approach to shed new light in the special situation where the control variable is assumed to be piecewise constant and thus is a non-convex control space.

\section{Very weak formulation of the inhomogenous Dirichlet problem}
For the analysis of the topology optimisation problem we recall the solution of the Dirichlet problem with $L^p$ boundary data. Let $\Omega\subset \VR^d$ be a bounded domain with smooth boundary ($C^{1,1}$ would be sufficient in our analysis). Let $g\in L^\infty(\partial \Omega)$ and $f\in L^2(\Omega)$ and consider the Dirichlet problem: find $u_{g,f}=u:\overline\Omega \to \VR$, such that
\begin{cases2}{dirichlet_g}
    -\Delta u && = f & \quad \text{ in } \Omega, \\
    u && = g & \quad \text{ on } \partial \Omega.
\end{cases2}
 Now in order to derive a very weak formulation we multiplying \eqref{dirichlet_g} with $\varphi \in C^2(\overline\Omega)$ and integrate by parts twice to obtain
\begin{equation}\label{eq:partial_integration_twice}
    \int_\Omega - u \Delta \varphi \;dx = - \int_{\partial \Omega} g \nabla \varphi\cdot n \; dS(x) + \int_\Omega f \varphi \;dx,
\end{equation}
where $n$ denotes the outward pointing normal vector along $\partial \Omega$. By density \eqref{eq:partial_integration_twice} also holds true for all $\varphi \in W^{2,q}(\Omega)$ for $1\le q<\infty$ and is well-defined by the trace theorem. Now to obtain the very weak form we fix $v\in L^q(\Omega)$,  and define $z_v\in W^{1,p}_0(\Omega)$, $\frac{1}{p}+\frac{1}{q}=1$, by 
\begin{equation}\label{eq:zv}
    \int_\Omega \nabla z_v \cdot \nabla \varphi \;dx = \int_\Omega v \varphi \;dx \quad \text{ for all } \varphi \in W^{1,q}_0(\Omega). 
\end{equation}
According to \cite[Chapter 3, Theorem~5.4]{b_CHWU_1998a} for each $v\in L^p(\Omega)$, there is a unique solution to \eqref{eq:zv} and we have the improved regularity $z_v\in W^{2,p}(\Omega)\cap W^1_0(\Omega)$, and there is a constant $C>0$, independent of $v$, such that
\begin{equation}\label{eq:estimate_zv}
    \|z_v\|_{W^{2,p}(\Omega)} \le C\|v\|_{L^p(\Omega)}.
\end{equation}
Hence using $\varphi=z_v$ with $v\in L^p(\Omega)$, $p>d$ as test function in \eqref{eq:partial_integration_twice} and noting $-\Delta z_v = v$, we obtain the very weak formulation of \eqref{dirichlet_g} described in the following definition.

\begin{definition}[very weak formulation]
    Given $g\in L^\infty(\partial \Omega)$ and $f\in L^2(\Omega)$ the very weak form of \eqref{dirichlet_g} reads: find $u_{g,f}\in L^p(\Omega)$ with $\frac{1}{p^\prime} + \frac{1}{p} =1$, such that
\begin{equation}\label{eq:very_weak_solution}
    \int_\Omega u_{g,f} v \;dx = - \int_{\partial \Omega} g \nabla z_v\cdot n \; dS(x) + \int_\Omega f z_v \;dx \quad \text{ for all } v\in L^p(\Omega),
\end{equation}
where $z_v$ is the unique solution to \eqref{eq:zv} for $v\in L^p(\Omega)$.
\end{definition}

Next we recall some standard a-priori estimates for the solution $u_{g,f}$ in terms of the data $g$ and $f$. 

\begin{theorem}\label{thm:regularity} 
\begin{itemize}
    \item[(a)] Let $p\ge 2$. For every $g\in L^\infty(\partial\Omega)$ and $f\in L^2(\Omega)$ there exists a unique weak solution $u_{g,f}\in L^p(\Omega)$ to \eqref{eq:very_weak_solution}. Moreover, there is a constant $C>0$, independent of $g$ and $f$, such that
    \begin{equation}\label{eq:apriori_ug2_2d}
        \|u_{g,f}\|_{L^p(\Omega)} \le 
        C(\|g\|_{L^\infty(\partial\Omega)}+\|f\|_{L^2(\Omega)}). 
    \end{equation}

\item[(b)] Set $f:=0$. For every $g\in L^1(\partial\Omega)$, there exists a unique weak solution $u_{g,0}\in L^p(\Omega)$, $p\in [1,\frac{d}{d-1})$ to \eqref{eq:constraint_dirichlet}. Moreover, there is a constant $C>0$, independent of $g$, such that
    \begin{equation}\label{eq:apriori_ug}
        \|u_{g,0}\|_{L^p(\Omega)} \le C\|g\|_{L^1(\partial\Omega)}.
    \end{equation}
\item[(c)] Set $f:=0$. For $g:= \delta_{x_0}$, where $\delta_{x_0}:C(\partial \Omega)\to \VR$ denotes the Dirac measure at $x_0$ there is a unique solution $u_{g,0}\in L^q(\Omega)$, $1\le q<\frac{d}{d-1}$ to \eqref{eq:very_weak_solution}.
\end{itemize}
\end{theorem}
\begin{proof}
   ad (a):  Let $p\ge 2$ and $q:= p/(p-1)$ and define the linear mapping $G:L^q(\Omega) \to \VR$ by 
\begin{equation}
   G(v):=  - \int_{\partial \Omega} g \nabla z_v\cdot n \; dS(x) + \int_\Omega f z_v \;dx.
\end{equation}
We first estimate as follows:
\begin{equation}\label{eq:estimate_g}
    |G(v)| \le \|g\|_{L^\infty(\partial \Omega)} \|\nabla z_v\|_{L^1(\partial \Omega)} + \|f\|_{L^2(\Omega)}\|z_v\|_{L^2(\Omega)} \le C\|v\|_{L^q(\Omega)}. 
\end{equation}
So we only need to estimate the terms $\|\nabla z_v\|_{L^1(\partial \Omega)}$ and $\|z_v\|_{L^2(\Omega)}$. The first term can be estimated using the trace theorem and \eqref{eq:estimate_zv}:
 for all $\ell \ge 2$ with $1+\frac{1}{\ell }\le q$ we have
 \begin{equation}\label{eq:estimate_partial_omega}
    \int_{\partial \Omega} |\nabla z_v|\;dS(x) \le C\|z_v\|_{W^{2,1+\frac{1}{p}}(\Omega)} \le C\|v\|_{L^{1+\frac{1}{\ell }}(\Omega)} \le C\|v\|_{L^{q}(\Omega)}. 
\end{equation}
The second term can be estimated as follows:
\begin{equation}\label{eq:estimate_zv_first_proof}
    \|z_v\|_{L^2(\Omega)} \le C\|z_v\|_{W^{2,p}(\Omega)} \le C\|v\|_{L^p(\Omega)}.
\end{equation}
Finally estimating \eqref{eq:estimate_g} with \eqref{eq:estimate_partial_omega} and \eqref{eq:estimate_zv_first_proof}:
\begin{equation}
    |G(v)| \le C(\|g\|_{L^\infty(\partial \Omega)} + \|f\|_{L^2(\Omega)})\|v\|_{L^q(\Omega)}, \quad\text{ for all } v\in L^q(\Omega).  
\end{equation}
Since $G$ is continuous on $L^q(\Omega)$ we find $u=u_{g,f}\in L^p(\Omega)$ by the Riesz representation theorem, such that
\begin{equation}
    \int_\Omega u v \;dx = G(v) \quad \text{ for all } v\in L^q(\Omega). 
\end{equation}
Moreover,
\begin{equation}
    \|u\|_{L^p(\Omega)} = \sup_{\substack{v\in L^{p^\prime}(\Omega)\\ \|v\|_{L^{p^\prime}(\Omega)}=1}}|G(v)| \le C(\|g\|_{L^\infty(\partial \Omega)} + \|f\|_{L^2(\Omega)}).
\end{equation}

ad (b): Let $p\in [1,\frac{d}{d-1})$ and note that  $q:= p/(p-1)>d$. We estimate $G$ as follows:
\begin{equation}
    |G(v)|\le \|g\|_{L^1(\Omega)}\|\nabla z_v\|_{C(\partial \Omega)}.
\end{equation}

Thus the Sobolev inequality yields $\|z_v\|_{C^1(\overline\Omega)} \le C\|z_v\|_{W^{2,p^\prime}(\Omega)}\le C\|v\|_{L^{q}(\Omega)}$. Therefore we can estimate in this case:
\begin{equation}
    |G(v)| \le \|z_v\|_{C^1(\overline \Omega)}\|g\|_{L^1(\Omega)} \le C\|g\|_{L^1(\Omega)}\|v\|_{L^p(\Omega)}. 
\end{equation}
Hence $G$ is a continuous linear functional on $L^{p^\prime}(\Omega)$ and thus by Riesz representation theorem we find a unique $u=u_{g,0}\in L^p(\Omega)$, such that
\begin{equation}\label{eq:equation_G}
    \int_\Omega u v\;dx = G(v) \quad \text{ for all } v\in L^{p^\prime}(\Omega).
\end{equation}
It follows:
\begin{equation}
    \|u\|_{L^p(\Omega)} = \sup_{\substack{v\in L^{p^\prime}(\Omega)\\ \|v\|_{L^{p^\prime}(\Omega)}=1}}|G(v)| \le C \|g\|_{L^1(\Omega)},
\end{equation}

ad (c): The prove of (c) is similar to (b). For $p\in [1,\frac{d}{d-1})$ define $G:L^{p'}(\Omega) \to \VR$ by 
\begin{equation}
    G(v) := -\delta_{x_0}(\nabla z_v \cdot n) = \nabla z_v(x_0) \cdot n(x_0). 
\end{equation}
The from $\|z_v\|_{C^1(\overline\Omega)} \le C\|z_v\|_{W^{2,p^\prime}(\Omega)}\le C\|v\|_{L^{p^\prime}(\Omega)}$ we see that $G$ is continuous. Hence by Riesz we obtain that there exists a unique $u_{g,0}\in L^p(\Omega)$ solving \eqref{eq:equation_G}.

\end{proof}

\section{Piecewise constant Dirichlet control problem as a multimaterial topology optimisation problem}
Throughout this paper we assume that $\Omega\subset\VR^d$, $d\in \{2,3\}$ is a smooth and bounded domain. We define the set of admissible set of designs by 
\begin{equation}
    \Ca(\partial \Omega):= \left\{(S_1,\ldots, S_M):\; S_1,\ldots, S_M\subset \partial \Omega \text{ open },\;  S_i\cap S_j=\emptyset\;  \text{ for } i\ne j, \; \partial \Omega = \overline{\cup_{i=1}^MS_i}  \right\}.
\end{equation}
Moreover, for given vectors $a,b\in \VR^M$ with components $a_i,b_i$, we define the set of admissible control values $\Cb\subset \VR^M$ by 
\begin{equation}
    \Cb_{a,b}:= \{x\in \VR^M:\; a_i\le x_i \le b_i, \quad \text{ for } i=1,\ldots,M \}.
\end{equation} 
With each design $S=(S_1,\ldots, S_M)\in \Ca(\partial \Omega)$ and control values $\alpha=(\alpha_1,\ldots,\alpha_M)^\top\in \Cb_{a,b}$, we associate the piecewise constant function $\alpha_S:\Omega\to \VR$:
\begin{equation}
    \alpha_S := \sum_{i=1}^M \alpha_i \chi_{S_i},
\end{equation}
which will act as the boundary control variable.  

\begin{definition}
We equip the set $\Ca(\partial \Omega)$ with the $L^1(\partial \Omega)$ topology: we say that the sequence of shapes $(S_n)$,  $S_n = (S_1^n,\ldots, S_M^n)\in \Ca(\partial\Omega)$ converges to $S=(S_1,\ldots, S_M)\in \Ca(\partial \Omega)$ if and only if
\begin{equation}\label{eq:topology_AD}
    \chi_{S_i^n} \underset{n\to\infty}{\longrightarrow} \chi_{S_i} \quad\text{ in } L^1(\partial\Omega). 
\end{equation}
Note that the $L^1(\partial \Omega)$-topology on $\Ca(\partial \Omega)$ is equivalent to the $L^p(\partial \Omega)$-topology for $p\in [1,\infty)$. This follows from the identity $\|\chi_{S_i^n}-\chi_{S_i}\|_{L^p(\partial\Omega)}=\|\chi_{S_i^n}-\chi_{S_i}\|_{L^1(\partial\Omega)}^{1/p}$. 
\end{definition}

For $S=(S_1,\ldots,S_M)\in \Ca(\partial\Omega)$ and $\alpha\in \Cb_{a,b}$, we consider: find $u=u_{S,\alpha}:\overline\Omega\to \VR$, such that
\begin{align}\label{eq:constraint_dirichlet}
\begin{split}
    -\Delta u& = f \quad \text{ in } \Omega, \\
            u& = \sum_{i=1}^M \alpha_i \chi_{S_i} \quad \text{ on } \partial \Omega.
        \end{split}
\end{align}

\begin{problem}
We consider the minimisation of the shape function
\begin{equation}
    \Cj(S) := \inf_{\alpha \in \Cb_{a,b}} J(S,\alpha)
\end{equation}
over all designs $S:=(S_1,\ldots, S_M)\in \Ca(\Omega)$, where 
\begin{align}\label{eq:cost_J_}
    J(S,\alpha):= \int_\Omega (u_{S,\alpha}-u_{\reff})^2\;dx + \lambda |\alpha|^2,
\end{align}
and $u_{S,\alpha}\in L^{p}(\Omega)$ is the very weak solution to (that is \eqref{eq:very_weak_solution} with $g:= \alpha_{S}$):
\begin{equation}
    \int_\Omega u_{S,\alpha} v \;dx = - \int_{\partial \Omega} \alpha_S \nabla z_v\cdot n \; dS(x) + \int_\Omega f z_v \;dx \quad \text{ for all } v\in L^{p^\prime}(\Omega).
\end{equation}
Here $|\cdot|$ denotes the Euclidean norm on $\VR^M$. In view of Theorem~\ref{thm:regularity} and $\alpha_S\in L^\infty(\partial\Omega)$, we have that 
\begin{equation}
    u_{S,\alpha}\in L^p(\Omega) \quad \text{ for } 1\le p <\infty.
\end{equation}
\end{problem}

\paragraph{Topological perturbation and topological state derivative}

Let $B=B_1(0)\subset \VR^{d-1}$ denote the Euclidean unit ball in $\VR^{d-1}$ centered at the origin. Denote by $\exp$ the exponential map of the manifold $\Cm:= \partial \Omega$. Then the we define for $x_0\in \Cm$:
\begin{equation}
    \omega_\eps(x_0) := \exp_{x_0}(E(\eps B))\subset M,
\end{equation}
where $E:T_{x_0}\Cm \to \VR^{d-1}$ is an isomorphism identifying $T_{x_0}\Cm$ with $\VR^{d-1}$. We have for the surface area (see \cite{a_GAST_2020b}) $|\omega_\eps(x_0)| = \eps^{d-1}|\exp_{x_0}(B)| + o(\eps^{d-1})$.

In the following definitions we let $S=(S_1,\ldots, S_M)\in \Ca(\partial\Omega)$. We consider the following topological perturbations.

\begin{definition}
Let $i,j\in \{1,\ldots, M\}$, $i\ne j$. For the point $x_0\in S_i$ we define the perturbation of $(S_i,S_j)$ by 
\begin{equation}
    S^{ij}_\eps(x_0):= (S_1,\ldots, S_i\setminus \omega_\eps(x_0), \ldots, S_j\cup \omega_\eps(x_0), \ldots, S_M).
\end{equation}
\end{definition}

\begin{remark}
Note that for $M=2$ and $S:=S_1$ and $S_2=\partial\Omega\setminus\overline S$ we have 
\begin{equation}
    S^{12}_\eps(x_0) = (S\setminus\overline{\omega_\eps(x_0)}, S\cup \overline \omega_\eps(x_0)).
\end{equation}
\end{remark}

\begin{definition}
    Let $\alpha\in \Cb_{a,b}$ and $i,j\in \{1,\ldots, M\}$. The topological state derivative $U^{ij}_{x_0}$ at $x_0\in S_i$ of $u_{S,\alpha}$ is defined by 
    \begin{equation}
        U^{ij}_{x_0}:= \lim_{\eps\searrow 0} \frac{u_{S_\eps^{ij}(x_0),\alpha}-u_{S,\alpha}}{|\omega_\eps|},
    \end{equation}
    where the limit has to be understood in the weak topology of the space $L^p(\Omega)$ for $p\in [1,\frac{d}{d-1})$.
\end{definition}

\begin{theorem}[topological state derivative]\label{thm:topological_state_derivative}
    Given $\alpha\in \Cb_{a,b}$ and $S=(S_1,\ldots, S_M)\in \Ca(\partial \Omega)$, we denote
    \begin{equation}
        U_{ij}^\eps:= \frac{u_{S_\eps^{ij}(x_0),\alpha}-u_{S,\alpha}}{|\omega_\eps|}.
    \end{equation}
\begin{itemize}
    \item[(a)] We have for all $p\in [1,\frac{d}{d-1})$:
    \begin{equation}
        U_{ij}^\eps \underset{\eps\searrow0}{\rightharpoonup} U^{ij}_{x_0,\alpha} \quad \text{ weakly in } L^p(\Omega),
    \end{equation}
where the topological state derivative $U^{ij}_{x_0}$ of $S \mapsto u_{S,\alpha}$ at $x_0\in S_j$, is given by the unique solution $U^{ij}_{x_0,\alpha}\in L^p(\Omega)$ to
\begin{equation}\label{eq:TSD_Uij}
        \int_\Omega U^{ij}_{x_0,\alpha} v \;dx = (\alpha_i-\alpha_j)\nabla z_v(x_0)\cdot n(x_0) \quad \text{ for all } v\in L^{p'}(\Omega). 
    \end{equation}
\item[(b)] Let $(\alpha_n)$, $\alpha_n\in \Cb_{a,b}$ be a sequence that converges to a vector $\alpha\in \Cb_{a,b}$ and $(\eps_n)$ is a null-sequence with $\eps_n\searrow 0$ as $n\to \infty$, then we have for all $p\in [1,\frac{d}{d-1})$:
    \begin{equation}
        \frac{u_{S_\eps^{ij}(x_0),\alpha_n}-u_{S,\alpha}}{|\omega_\eps|}  \underset{\eps\searrow0}{\rightharpoonup} U^{ij}_{x_0,\alpha} \quad \text{ weakly in } L^p(\Omega).
    \end{equation}
\end{itemize}
\end{theorem}
\begin{proof}
    First note that $\alpha_{S_\eps^{ij}}-\alpha_S = -(\alpha_i-\alpha_j)\chi_{\omega_\eps}$. Hence taking the difference of the equations defining $u_{S^{ij}_\eps,\alpha}$ and $u_{S,\alpha}$ yields
    \begin{equation}\label{eq:difference_uij}
        \int_{\Omega} (u_{S^{ij}_\eps,\alpha}-u_{S,\alpha})v \;dx = - \int_{\partial \Omega} (\alpha_{S_\eps^{ij}}-\alpha_S) \nabla z_v\cdot n \; dS(x) = \int_{\partial \Omega} (\alpha_i-\alpha_j) \chi_{\omega_\eps} \nabla z_v\cdot n \;dS(x).
    \end{equation}
    Hence applying Theorem~\ref{thm:regularity}, item (b) with $g:= -|\omega_\eps|^{-1}(\alpha_i-\alpha_j) \chi_{\omega_\eps}\in L^\infty(\partial \Omega)$ and $f\in L^2(\Omega)$ yields after dividing \eqref{eq:difference_uij} by $|\omega_\eps|$: 
\begin{equation}
    \|U^{ij}_\eps\|_{L^p(\Omega)}\le C\| -(\alpha_i-\alpha_j) |\omega_\eps|^{-1}\chi_{\omega_\eps}\|_{L^1(\partial \Omega)} = C|\alpha_i-\alpha_j|.
\end{equation}
Here the constant $C>0$ is independent of $\eps >0$ and $\alpha_i,\alpha_j$. This shows that $U^{ij}_\eps$ is bounded independently of $\eps$ in $L^p(\Omega)$. Hence for every null-sequence $(\eps_n)$, $\eps_n\searrow 0$ for $n\to \infty$, we find a subsequence, denoted the same, such that $U^{ij}_\eps \rightharpoonup U$ weakly in $L^p(\Omega)$ for $U\in L^p(\Omega)$. On the other hand dividing \eqref{eq:difference_uij} by $|\omega_\eps|$, $\eps>0$, we find that 
\begin{equation}\label{eq:Uij_epsn}
    \int_{\Omega} U^{ij}_{\eps_n}v \;dx =  \frac{1}{|\omega_{\eps_n}|}\int_{\omega_{\eps_n}} (\alpha_i-\alpha_j) \nabla z_v\cdot n\;dS(x).
\end{equation} 
Since $v\in L^{p'}(\Omega)$ we have that $\nabla z_v\cdot n \in C(\partial \Omega)$ and thus passing to the limit $n\to \infty$ yields
\begin{equation}\label{eq:Uij_epsn_limit}
    \int_{\Omega} U v \;dx =  (\alpha_i-\alpha_j) \nabla z_v(x_0)\cdot n(x_0) \quad \text{ for all } v\in L^{p'}(\Omega). 
\end{equation} 
According to Theorem~\ref{thm:regularity}, item(c), this equation admits a unique solution for $g:= (\alpha_i-\alpha_j)\delta_{x_0}$ in $L^p(\Omega)$ for $p\in [1,\frac{d}{d-1})$. Therefore the whole family $(U_\eps^{ij})$ weakly converges to $U$ in $L^p(\Omega)$.

If now $(\alpha_n)$ sequence converges to $\alpha$, then we can pass to the limit \eqref{eq:Uij_epsn} and obtain again \eqref{eq:Uij_epsn_limit}.
\end{proof}

\paragraph{Main result on the topological derivative}

\begin{definition}
Let $S=(S_1,\ldots, S_M)\in \Ca(\partial \Omega)$. Let $x_0\in S_i$ and let $i,j\in \{1,\ldots, M\}$, $i\ne j$, $i<j$. The topological derivative of $\Cj$ at $S$ in the direction $(x_0,S_i,S_j)$ is defined by  
    \begin{equation}
        D^{ij}\Cj(S)(x_0) =  \lim_{\eps\searrow 0} \frac{\Cj (S_\eps^{ij})- \Cj(S)}{|\omega_\eps(x_0)|}.
    \end{equation}
    Here the indices $i,j$ correspond to the perturbations of $S_i$,$S_j$, respectively.  
\end{definition}

\begin{lemma}
    For $S\in \Ca(\Omega)$ the unique minimiser $\alpha^*\in \Cb_{a,b}$ of 
\begin{equation}
    \inf_{\alpha\in \Cb_{a,b}} J(S,\alpha)
\end{equation}
is given by 
\begin{equation}
    \alpha^*_i = P_{[a_i,b_i]}\left(\frac{1}{2\lambda} \int_{S_i} \nabla p\cdot n \; dS(x) \right),
\end{equation}
where $p\in H^1_0(\Omega)$ solves the adjoint equation
where the adjoint $p\in H^1_0(\Omega)$ is given by 
\begin{equation}\label{eq:adjoint_equation}
    \int_\Omega \nabla p \cdot \nabla \varphi \;dx = - \int_\Omega 2(u_{S,\alpha}-u_{\reff})\varphi\;dx \quad \text{ for all } \varphi \in H^1_0(\Omega).
\end{equation}
Here $P_{[a,b]}(x) = \min\{\max\{a,x\},b\}$ denotes the projection of $x$ into the interval $[a,b]$.
\end{lemma}
\begin{proof}
    This result follows readily by standard arguments of optimal control theory and hence left to the reader; see \cite[Chapter 2]{b_TR_2010a}. 
\end{proof}

\begin{theorem}\label{thm:main}
    The topological derivative $D^{ij}\Cj(S_1,\ldots,S_M)$ at $x_0\in \partial \Omega\setminus(\partial S_1\cup\cdots\cup \partial S_M)$ is given by 
    \begin{equation}\label{eq:topo_cJ_ij}
    D^{ij}\Cj(S_1,\ldots,S_M)(x_0) = -(\alpha_i^*(x_0)-\alpha_j^*(x_0)) \nabla p(x_0)\cdot n(x_0)
\end{equation}
where
\begin{equation}
    \alpha_i^* = P_{[a_i,b_i]}\left(\frac{1}{2\lambda} \int_{S_i} \nabla p\cdot n \; dS(x) \right)
\end{equation}
where $p\in H^1_0(\Omega)$ is the solution to \eqref{eq:adjoint_equation}.
\end{theorem}
\begin{proof}
 Let $S=(S_1,\ldots, S_M)\in \Ca(\Omega)$ be fixed. Throughout the proof we set $\omega_\eps:= \omega_\eps(x_0)$.  Using Theorem~\ref{thm:topological_state_derivative}, items (a) and (b), we obtain
\begin{equation}
    \lim_{\eps\searrow 0}\frac{J (S_\eps^{ij},\alpha)- J(S,\alpha)}{|\omega_\eps|} = \lim_{\eps\searrow0}\int_{\Omega} U_\eps^{ij} (u_{S_\eps,\alpha} + u_{S,\alpha} - 2u_\reff)\;dx = \int_\Omega 2(u_{S,\alpha} - u_\reff) U^{ij}_{x_0}\;dx
\end{equation}
for every $\alpha\in \Cb_{a,b}$. For $\eps>0$  we denote by  $\alpha^*_\eps\in \Cb_{a,b}$ the minimiser of 
\begin{equation}\label{eq:opt_eps}
\Cj(S_\eps^{ij}) = \min_{\alpha\in \Cb_{a,b}} J(S_\eps^{ij},\alpha) = J(S_\eps^{ij},\alpha_\eps^*).
\end{equation}
Similarly we denote by $\alpha^*\in \Cb_{a,b}$ the minimiser of 
\begin{equation}\label{eq:opt_0}
\Cj(S) = \min_{\alpha\in \Cb_{a,b}} J(S,\alpha) = J(S,\alpha^*).
\end{equation}
Since $\alpha^*$ is the unique minimiser of the previous problem it is easy to see that $\alpha^*_\eps\to \alpha^*$ as $\eps\searrow 0$. 

On the one hand from \eqref{eq:opt_eps} we have $\Cj(S) \le J(S,\alpha^*_\eps)$ and therefore
\begin{equation}
    \Cj(S_\eps^{ij}) - \Cj(S) \ge  J(S_\eps^{ij},\alpha_\eps^*) - J(S,\alpha_\eps^*).  
\end{equation}
Dividing by $|\omega_\eps|$, we obtain 
\begin{equation}\label{eq:limsup}
\limsup_{\eps\searrow 0} \frac{\Cj(S_\eps^{ij}) - \Cj(S)}{|\omega_\eps|} \ge  \limsup_{\eps\searrow 0} \frac{J(S_\eps^{ij},\alpha_\eps^*) - J(S,\alpha_\eps^*)}{|\omega_\eps|}.
\end{equation}
On the other hand from \eqref{eq:opt_0} we have $\Cj(S_\eps^{ij})\le J(S_\eps^{ij},\alpha^*)$ and thus
\begin{equation}
    \Cj(S_\eps^{ij})-\Cj(S) \le j(S_\eps^{ij},\alpha^*) - J(S,\alpha^*).
\end{equation}
Dividing by $|\omega_{\eps}^{ij}|$ and taking the $\liminf_{\eps\searrow 0}$ yields
\begin{equation}\label{eq:liminf}
    \liminf_{\eps\searrow 0}\frac{\Cj(S_\eps^{ij})-\Cj(S)}{|\omega_\eps|} \le \liminf_{\eps\searrow 0} \frac{j(S_\eps^{ij},\alpha^*) - J(S,\alpha^*)}{|\omega_\eps|}.
\end{equation}
Now we compute the right hand sides of \eqref{eq:liminf} and \eqref{eq:limsup}.  
By definition of the $\limsup$ and since $(\alpha_\eps^*)$ is bounded, we find a sequence $(\eps_k)$ converging to zero such that
\begin{equation}
    \limsup_{\eps\searrow 0} \frac{J(S_\eps^{ij},\alpha_\eps^*) - J(S,\alpha_\eps^*)}{|\omega_\eps|} = \lim_{k\to\infty} \frac{J(S_{\eps_k}^{ij},\alpha_{\eps_k}^*) - J(S,\alpha_{\eps_n}^*)}{|\omega_{\eps_k}|}.
\end{equation}
By Theorem~\ref{thm:topological_state_derivative}, we obtain
\begin{align}
    \lim_{n\to\infty} \frac{J(S_\eps^{ij},\alpha_\eps^*) - J(S,\alpha_\eps^*)}{|\omega_\eps|} & = \lim_{n\to\infty}\int_{\Omega} U_\eps^{ij}(u_{S_\eps^{ij},\alpha,\eps} + u_{S,\alpha} - 2u_{\text{ref}})\;dx \\
                                                                                                   & = \int_\Omega 2U_{x_0,\alpha}^{ij}(u_{S,\alpha}-u_{\reff})\;dx. 
\end{align}
Similarly we check that 
\begin{equation}
    \liminf_{\eps\searrow 0} \frac{J(S_\eps^{ij},\alpha) - J(S,\alpha)}{|\omega_\eps|} = \int_\Omega 2U_{x_0,\alpha}^{ij}(u_{S,\alpha}-u_{\reff})\;dx.
\end{equation}
From \eqref{eq:limsup} and \eqref{eq:liminf}  it follows that 
\begin{equation}
\lim_{\eps\searrow 0} \frac{\Cj(S_\eps^{ij}) - \Cj(S)}{|\omega_\eps|} = \int_\Omega 2U_{x_0,\alpha^*}^{ij}(u_{S,\alpha}-u_{\reff})\;dx
\end{equation}

Finally, in order to establish \eqref{eq:topo_cJ_ij} we test \eqref{eq:adjoint_equation} with $U_{x_0,\alpha^*}^{ij}$ and \eqref{eq:TSD_Uij} (with $\alpha:=\alpha^*$) with $p$:
\begin{align}
    \int_\Omega 2U_{x_0,\alpha^*}^{ij}(u_{S,\alpha}-u_{\reff})\;dx & \stackrel{\eqref{eq:adjoint_equation}}{=} -\int_\Omega \nabla U_{x_0,\alpha^*}^{ij}\nabla p\;dx \\
                                                                   &\stackrel{\eqref{eq:TSD_Uij}}{=}  -(\alpha_i^*-\alpha_j^*) \nabla p(x_0)\cdot n(x_0).
\end{align}
This finishes the proof.

\end{proof}

\newpage
\section{Numerical algorithm}
In this section we recall the multi-material level-set algorithm introduced in Let $M\ge 2$ be an integer.
Given $M$ shapes $S_1,\ldots,S_M\in \Ca(\Omega)$ we represent the shapes with a vector valued level-set function $\psi=(\psi_1,\ldots, \psi_M):\Omega\to \VR$. We following level-set approach of \cite{a_GA_2020a}, which is a generalisation of the level-set method \cite{a_AMAN_2006a}.

\subsection{Generalised topological derivatives and levelsets}

Following \cite{a_GA_2020a} we represent the shape $S=(S_1,\ldots, S_M)$ by means of a level-set function $\psi:\partial \Omega \to \VR^{M-1}$ as follows. We assume that $\VR^{M-1}$ is divided into $M-1$ open and convex sector $s_1,\ldots, s_M$ such that $\VR^{M-1} = \cup_{\ell=1}^{M-1} \overline s_\ell$. Each sector $s_\ell$ is uniquely defined by $M-1$ hyperplanes $H_{\ell,1},\ldots, H_{\ell,\ell-1},H_{\ell,\ell+1}, H_{\ell,M-1}\subset \VR^{M-1}$. Moreover, each hyperplane $H_{i,j}$ is uniquely determined by unitary normal vector $n^{i,j}\in \VR^{M-1}$.  Each sector $s_\ell$ can be represented by
\begin{equation}
    s_\ell = \{x\in \VR^{M-1}: n^{j,\ell}\cdot x>0:\; \text{ for } j\in \{1,\ldots, M\}\setminus \{\ell\}\}.
\end{equation}
Further we define on each shape $s_\ell$ the function $T^\ell(S,\cdot):S_\ell \to \VR^{M-1}$ by 
\begin{equation}\label{eq:gen_topo}
    T^\ell(S,x)^\top := \begin{pmatrix}
  D\Cj^{\ell,1}(S)(x), \ldots, D\Cj^{\ell,\ell-1}(S)(x), D\Cj^{\ell,\ell+1}(S)(x), \ldots, D\Cj^{\ell,M}(S)(x)
    \end{pmatrix}^\top.
\end{equation}
We define the matrix $N^\ell\in \VR^{(M-1)\times (M-1)}$ by 
\begin{equation}
    (N^\ell)^\top := \begin{pmatrix}
        n^{1,\ell}, \cdots, n^{\ell-1,\ell},n^{\ell+1,\ell}, \cdots, n^{M-1,\ell} 
    \end{pmatrix}
\end{equation}
and also 
\begin{equation}
    G_\ell(S,x):= (N^{\ell})^{-1}T^\ell(S,x). 
\end{equation}
Finally we define $T(S,\cdot):\partial \Omega \setminus\{S_1\cup\cdots\cup S_M\}$ by 
\begin{equation}
    T(S,x):= G_\ell(S,x) \quad x\in S_\ell
\end{equation}
so that $T$ is defined piecewise on each sector $s_\ell$. 

\subsection{Level-set update}
Let now $S^0 = (S_1^0,\ldots, S_M^0)\in \Ca(\partial \Omega)$ an initial shape and assume that $\psi_0:\partial \Omega \to \VR^{M-1} $ is such that
\begin{equation}
    x\in S_\ell \quad \Leftrightarrow \quad \psi_0(x)\in s_\ell.
\end{equation}
Then a level-set update is performed via the formula 
\begin{equation}
    \psi_1(x):= s \psi_0(x) + (1-s) T(S_0,x)
\end{equation}
with some step size $s\in (0,1)$. In case $s=1$ one would replace the function $\psi_1$ by $T(S_0,\cdot)$. Moreover, even if the function $\psi_0$ is continuous, due to the piecewise definition of $T$, the update function $\psi_1$ does not have to be continuous. Numerically we will work with a piecewise constant level-set function, which then simply needs to be added (with some factors) to $T$.  

For $M=2$ one can simply use $s_1=(-\infty,0)$ and $s_2=(0,\infty)$ which leads to the usual level-set algorithm. In particular for the first two iterations $i=0,1$:
\[
    S^i_1 = \{x\in \VR^2:\; \psi_i(x)<0\}, \quad S^i_2 = \partial \Omega\setminus \overline{S}^i_1 = \{x\in \VR^2:\; \psi_i(x)>0\}.
\]

For $M=3$ (here the sectors are subsets of $\VR^2$) the normal vectors $n^{ij}$, $i,j=1,\ldots, 3$, $i\ne j$ defining the sectors $s_1,s_2,s_3$ can be chosen as
\begin{equation}
    n^{12} = \begin{pmatrix}
        1/\sqrt{2}\\
        -1/\sqrt{2}
        \end{pmatrix}, \quad n^{13} = \begin{pmatrix}
        1 \\0
        \end{pmatrix}, \quad n^{23} = \begin{pmatrix}
  0 \\ 1
    \end{pmatrix}.
\end{equation}
Note that $n^{ij} = -n^{ji}$, so that all normal vectors are determined for $M=3$. In our case we have
\begin{equation}
    N^1 = - \begin{pmatrix}
 1/\sqrt2 & - 1/\sqrt2 \\
 1 & 0 
    \end{pmatrix},\quad 
    N^2 = \begin{pmatrix}
        1/\sqrt{2} & -\sqrt{2} \\
 0 & -1
    \end{pmatrix}, \quad
    N^3 = \begin{pmatrix}
        1 & 0 \\
        0 & 1 
    \end{pmatrix}.
\end{equation}
Moreover,
\begin{equation}
T^1(S,x) = \begin{pmatrix}
    D\Cj^{2,1}(S)(x) \\
    D\Cj^{3,1}(S)(x)
\end{pmatrix},\quad 
T^2(S,x) = \begin{pmatrix}
    D\Cj^{1,2}(S)(x) \\
    D\Cj^{3,2}(S)(x)
\end{pmatrix}, \quad 
T^3(S,x) = \begin{pmatrix}
    D\Cj^{1,3}(S)(x) \\
    D\Cj^{2,3}(S)(x)
\end{pmatrix}
\end{equation}
and thus explicitly in our case
\begin{align}
    T^1(S,x) &= -\nabla p(x)\cdot n(x) \begin{pmatrix}
    \alpha_1^* - \alpha_2^* \\
    \alpha_1^* - \alpha_3^*
\end{pmatrix},  \quad 
T^2(S,x) = -\nabla p(x)\cdot n(x) \begin{pmatrix}
    \alpha_2^* - \alpha_1^* \\
    \alpha_2^* - \alpha_3^*
\end{pmatrix}, \\ 
T^3(S,x) & = -\nabla p(x)\cdot n(x) \begin{pmatrix}
    \alpha_3^* - \alpha_1^* \\
    \alpha_3^* - \alpha_2^*
\end{pmatrix}
\end{align}
We can now check $x\in S_\ell$ if and only if $\psi(x)\in s_\ell$ if and only if 
\begin{equation}
    \psi(x)\cdot n^{i\ell}>0 \quad \text{ for } i \in \{1,2,3\}, i\ne \ell.
\end{equation}

We close with the complete level-set algorithm we employ. We are not going to optimise over the values of $\alpha=(\alpha_1,\ldots, \alpha_M)$, but keep them fixed. Hence for fixed $\alpha\in \VR^M$ we minimise
\begin{equation}
    \tilde \Cj(S_1,\ldots, S_M) := J(\alpha,S_1,\ldots, S_M). 
\end{equation}
Its topological derivatives of $\tilde \Cj$ at $S=(S_1,\ldots, S_M)$ in direction $(x_0,S_i,S_j$) is given by 
\begin{equation}
D^{ij}\tilde \Cj(S)(x_0)  = - (\alpha_i-\alpha_j) \nabla p(x_0)\cdot n(x_0),
\end{equation}
This follows directly as a special case from our main Theorem~\ref{thm:main}. 

\begin{algorithm}
    \caption{\bf multi-material level-set algorithm}\label{alg:multi}
\begin{algorithmic}[1]
    \REQUIRE initial shape $S^0=(S_1^0,\ldots, S_M^0)$, very weak solution of state equation $u_0$ of \eqref{eq:constraint_dirichlet} on $S^0$, weak solution of adjoint equation \eqref{eq:adjoint_equation} on $S^0$, level-set function $\psi_0$, max iterations $N_{max}$, step size $N_{step}$.

\FOR { $ i = 1 : N_{max} $ } 
\State Compute state $u_i$ and adjoint state $u_i$ with respect to  $S^i = (S_1^i,\ldots, S_M^i)$.
\State Compute topological derivatives $T^\ell(S^i,x)$, $x\in S^i_\ell$ for $\ell=1,\ldots, M$. 
\State Compute $G_\ell(S^i,x) = (N^\ell)^{-1}T^\ell(S^i,x)$, $x\in S^i_\ell$ for $1,\ldots, M$.
\FOR {$k=1:N_{step}$}
\State Compute $\psi_{i+1}(x):= s\psi_i(x) + (1-s) G_\ell(S^i,x)$, $x\in S^i_\ell$, $\ell=1,\ldots,M $ 
\State Compute $S^{i+1}_\ell := \{x\in \partial \Omega: \; n^{r\ell}\cdot \psi_{i+1}^\ell(x)>0 \text{ for } r\in \{1,\ldots, M\}\setminus \{\ell\}\}$.
\State Compute state $u_{i+1}$ and adjoint state $p_{i+1}$ for the domain $S^{i+1}$ to evaluate $\hat \Cj(S^{i+1})$. 
\IF{$ \hat \Cj( S^{i+1})<\hat \Cj(S^i)$}:
\State step accepted; break for-loop and go back to outer for-loop
\ELSE{} reduce step size $s$. 
\ENDIF
\ENDFOR
\ENDFOR
\end{algorithmic}
\end{algorithm}

\begin{remarks}
    \begin{itemize}
        \item As can be seen from the previous algorithm the update of $\psi_i$ is done piecewise, meaning that we do the update on each $S^i_\ell$. Numerically the functions $G_\ell$ are defined by piecewise $L^2$-finite elements, since it involves the gradient $\nabla p$ and the normal vector field $n$, which are both piecewise constant. The gradient $\nabla p$ is piecewise constant as $p$ is a linear $P1$-finite element function on a simplicial mesh and $n$ is piecewise constant since the mesh boundary $\partial \Omega$ consists of triangles. Accordingly we also discretise the level-set function $\psi$ in a piecewise constant $L^2$-finite element space. Better results can be obtained by representing $\psi$ in a $P1$ continuous finite element space, but then interpolation is needed to obtain a shape from the level-set function, since the "zero" level-set might cut triangles on the boundary $\partial \Omega$. Therefore, here, we follow this simplified approach, which yields shapes that are defined only on triangles of the boundary shape $\partial\Omega$. 
    \item The numerical computation of the very weak solution is done by projecting the piecewisely defined function $\sum_{\ell=1}^M \alpha_\ell \chi_{S_\ell^i}$ onto the $P1$ finite element space. According to \cite{a_BE_2004a} the very weak solution and the weak solution of the projected function coincide when we use $P1$ finite elements. Since now 
        the projected inhomogeneity $\sum_{\ell=1}^M \alpha_\ell \chi_{S_\ell^i}$ belongs to the continuous $P1$ finite element space, the standard homogenisation technique can 
        be used to solve inhomogenous Dirichlet problems. 
\end{itemize}
\end{remarks}

\newpage
\section{Numerical examples in NGSolve \cite{Schoeberl2014}}
\subsection{Setting and parameter selection}
In our numerical experiments we choose an  ellipsoid for the domain $\Omega$:
\[
    \Omega = \{(x_1,x_2,x_3)\in \VR^3:\; (x_1/a_1)^2 + (x_2/a_2)^2 + (x_3/a_3)^2 = 1 \}
\]
with $x,y,z$-axes of length $a_1 = 1, a_2 = a_1/2, a_3 = a_1$, respectively. We henceforth work with three shapes and therefore let $M=3$. Moreover, we consider the shape function
\begin{equation}
    \Cj(S_1,\ldots, S_4):= J(\alpha, S_1,\ldots, S_3), 
\end{equation}
where $J$ is defined in \eqref{eq:cost_J_} and we set $\lambda := 0$. 

We henceforth choose the values $\alpha_1 = 0.1$, $\alpha_2 = 10$ and $\alpha_3 = 3$. We consider the case $M=3$, so that the state reads:
\begin{align}
    \begin{split}
        -\Delta u & = 1 \quad \text{ in } \Omega, \\
        u & = \alpha_1 \chi_{S_1} + \alpha_2 \chi_{S_2} + \alpha_3 \chi_{S_3}  \quad \text{ on } \partial \Omega,
    \end{split}
\end{align}
where we also chose $f:=1$.  We examine two different shapes that we want to reconstruct. 

We initialise for all our example the level-set function $\psi$ as follows. We first solve for $z\in \VR^2$:
\begin{equation}
    \begin{pmatrix}
 n_{13}\\
 n_{23}
 \end{pmatrix}z = \begin{pmatrix}
1 \\ 1
\end{pmatrix},
\end{equation}
and set $\psi(x):= z$ for all $x\in \partial \Omega$. Then by construction for all $x\in \partial \Omega$, we have  $0<\psi(x)\cdot n_{13} = 1 $ and $0<\psi(x)\cdot n_{23}=1$, which means that $\partial \Omega = S_3$. 

The space $H^1(\Omega)$ is discretised with linear $P1$ finite elements. We discretise in all our examples with a uniform mesh with $201004$ degrees of freedom ($maxh=0.2$ in NGSolve).

\subsection{Two test examples}
\paragraph{Numerical results: reconstruction of two materials}
We define the shapes:
\begin{equation}
    S_1 := \{ (x,y,z) \in \partial \Omega:\; z^2+y^2 < 0.1  \}, \quad S_2 := \partial \Omega \setminus \overline S_1. 
\end{equation}
Moreover we solve 
\begin{align}
    \begin{split}
        -\Delta u_{\reff} &= 1 \quad \text{ in } \Omega\\
        u_{\reff} & = \alpha_1 \chi_{S_1} + \alpha_2 \chi_{S_2} \quad \text{ on } \partial \Omega. 
    \end{split}
\end{align}
We set $S_{opt}^a := (S_1,S_2,\emptyset,\emptyset)$, which is depicted in Figure~\ref{fig:two_phases_reference}. This set is the global minimum of the shape functional$\Cj$ with the reference function $u_{\reff}$ defined before.
We apply Algorithm~\ref{alg:multi} with $M=3$ and a step size rule where we half the step size in each step with does not yield a descent. The initial step size was $s_0=0.1$ in our example with a step size control. Several iterations are shown in Figure~\ref{fig:two_phases_iterations}. In Figure~\ref{fig:two_phases_reference} we see the reference solution $S_{opt}^a$ corresponding to
the reference solutions $u_{\reff}$. The initial shape has the cost function value $80.8963905152006$ and the shape at iteration 48 iteration has the cost function value $0.005590737594271838$.

\paragraph{Numerical results: reconstruction of three materials}
We define the shapes:
\begin{equation}
    S_1 := \{ (x,y,z) \in \partial \Omega:\; x<0 \;\text{ and }\; y<0  \}, \quad S_2 := \{(x,y,z) \in \partial \Omega:\; x <0 \; \text{ and  } \; y>0\},
\end{equation}
and 
\begin{equation}
    S_3 := \partial \Omega \setminus \overline{(S_1\cup S_2)}. 
\end{equation}
Moreover we solve 
\begin{align}
    \begin{split}
        -\Delta u_{\reff} &= 1 \quad \text{ in } \Omega\\
        u_{\reff} & = \alpha_1 \chi_{S_1} + \alpha_2 \chi_{S_2} + \alpha_3 \chi_{S_3} \quad \text{ on } \partial \Omega. 
    \end{split}
\end{align}
We set $S_{opt}^b:= (S_1,S_2,S_3,\emptyset)$, which is depicted in Figure~\ref{fig:three_phases_reference}. This set is the global minimum of the shape functional$\Cj$ with the reference function $u_{\reff}$ defined before. Again we run  Algorithm~\ref{alg:multi} with $M=3$. We use the initial step size $s=0.01$ and a step size rule to reduce the step size in case the cost function did not decrease. The initial shape is the same as in the previous numerical experiment. The initial cost functional value is now $15.063614565848008$. Note that this value is different from the previous example as the function $u_{\reff}$ is different. The final cost functional value at iteration $46$ is $0.006319718137496762$. The initial shape is depicted in Figure~\ref{fig:three_phases_reference} and several iterations are shown in Figure~\ref{fig:three_phases_iterations}.

\begin{remark}
Since the level-set function is discretised by piecewise constant functions the intersection of the blue phase (with value $\alpha_1$) and the red phase (with value $\alpha_2$) does not disappear entirely and an interpolation could improve this. 
\end{remark}


\begin{figure}
\centering
\begin{subfigure}{.49\textwidth}
  \centering
  \includegraphics[scale=0.4]{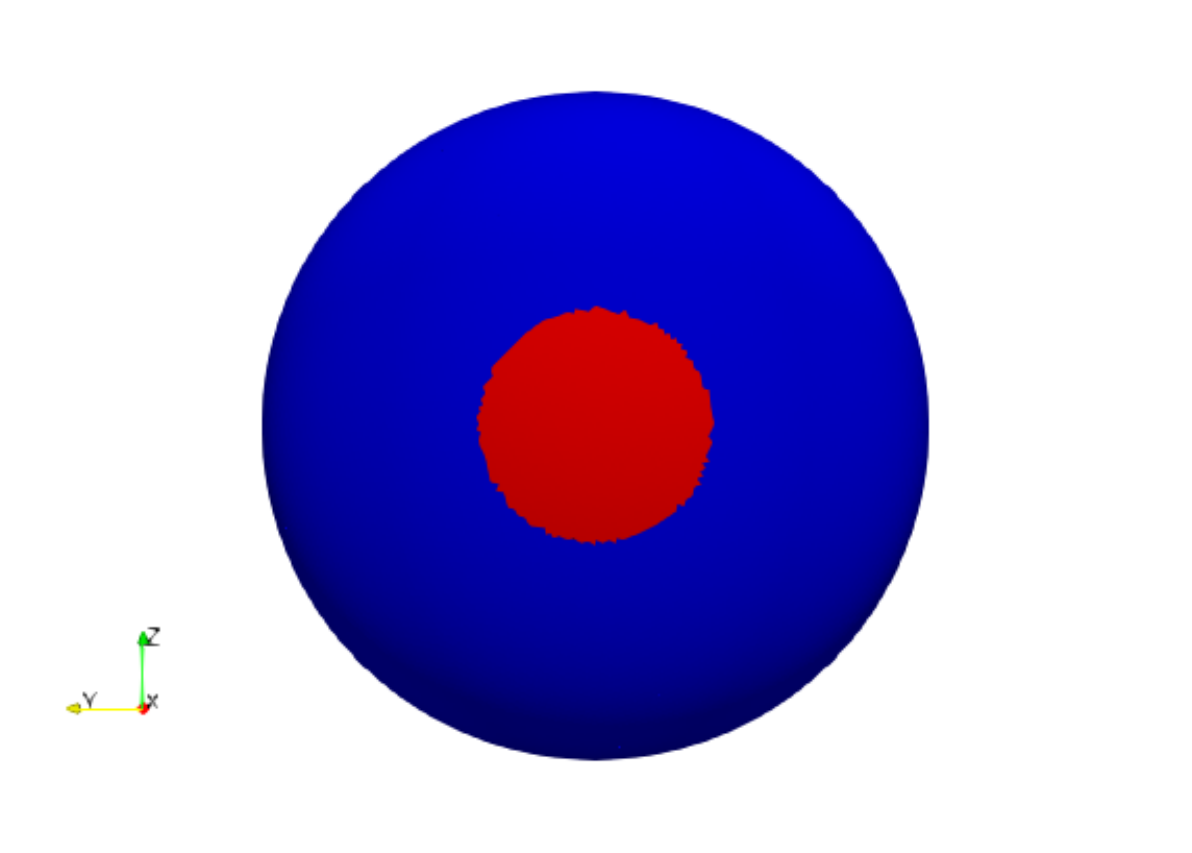}
  \caption{$y-z$ plane $+x$}
  \label{fig:two_phase_ref_yz2}
\end{subfigure}
\begin{subfigure}{.49\textwidth}
  \centering
  \includegraphics[scale=0.4]{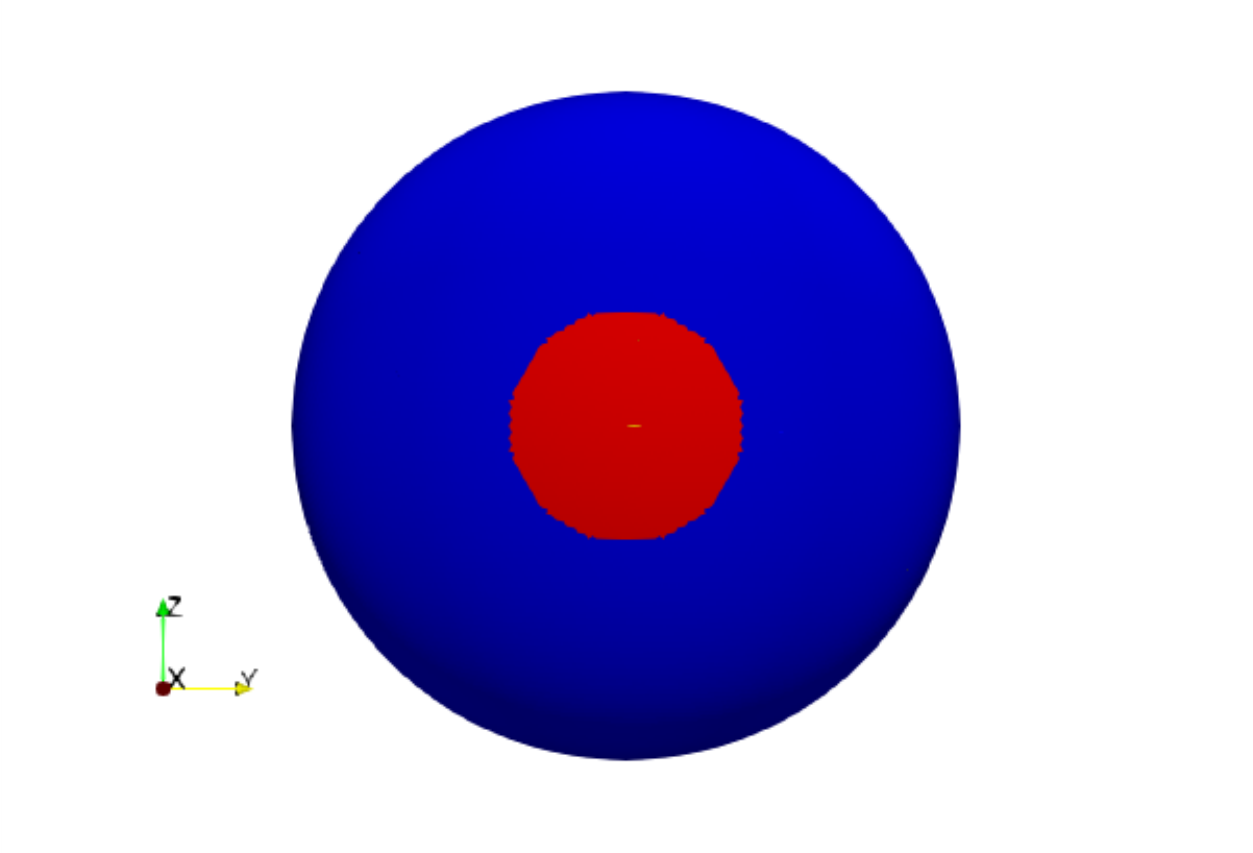}
  \caption{$y-z$ plane $-x$}
  \label{fig:two_phase_ref_yz}
\end{subfigure}
\begin{subfigure}{.4\textwidth}
  \centering
  \includegraphics[scale=0.4]{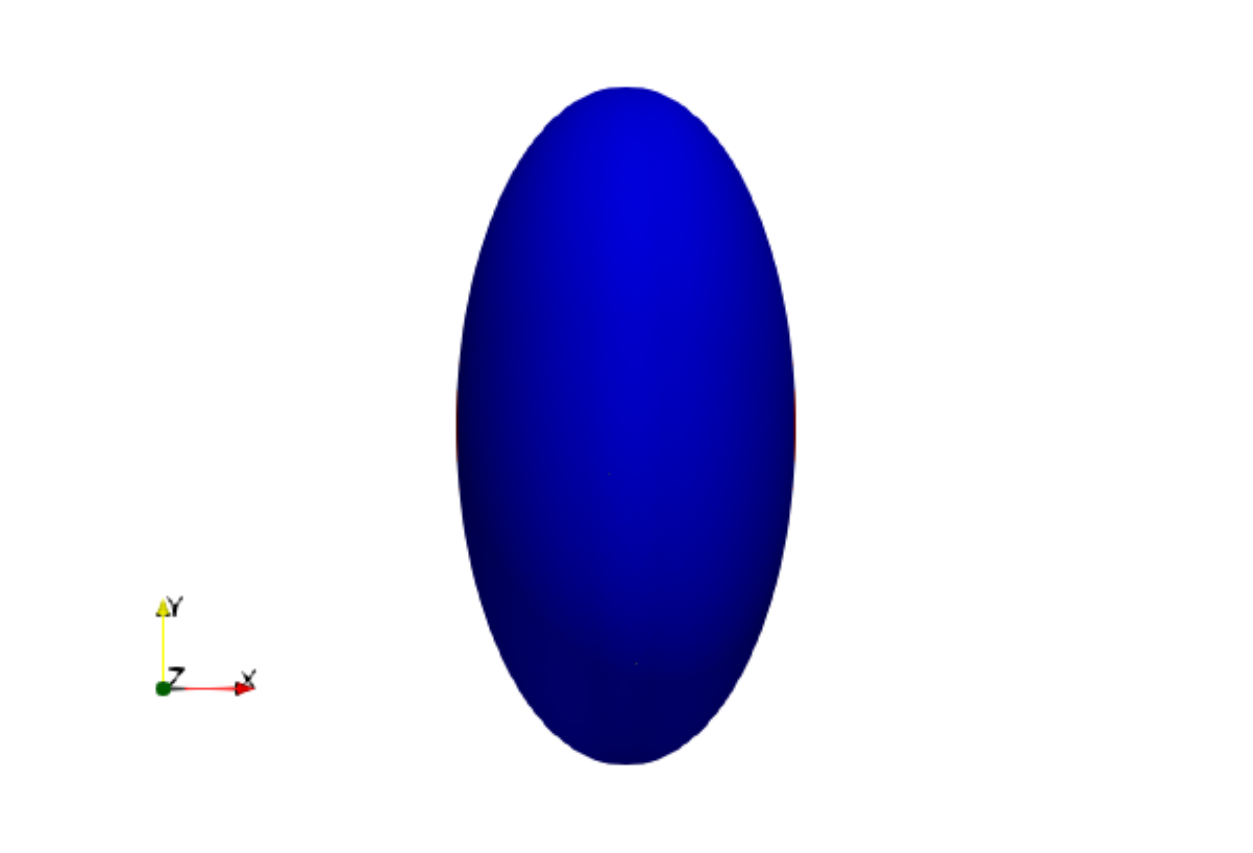}
\caption{$x-y$ plane}
  \label{fig:two_phase_ref_xy}
\end{subfigure}
\begin{subfigure}{.4\textwidth}
  \centering
  \includegraphics[scale=0.4]{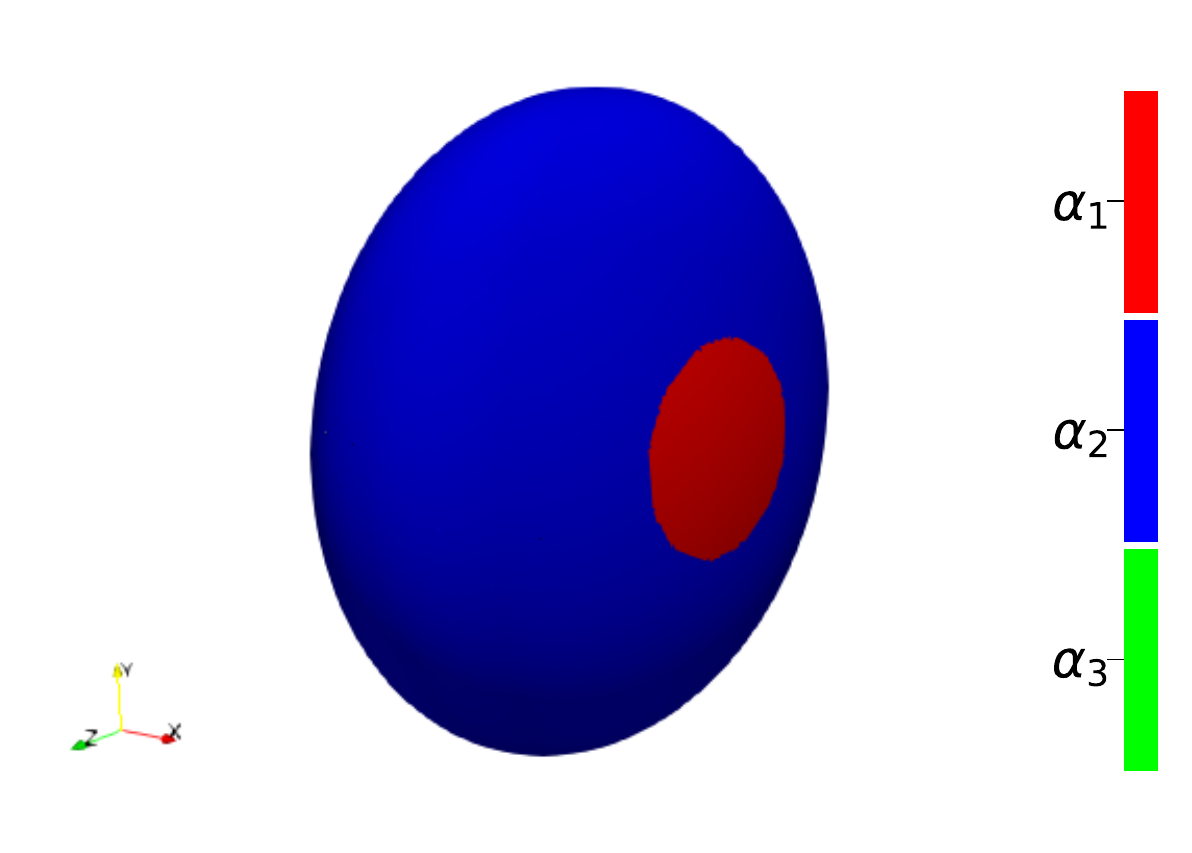}
  \caption{$x-y-z$-plane}
  \label{fig:two_phase_ref_xyz}
\end{subfigure}
\caption{Optimal shape $S_{opt}^a$}
\label{fig:two_phases_reference}
\end{figure}

\begin{figure}
\centering
\begin{subfigure}{.49\textwidth}
  \centering
  \includegraphics[scale=0.4]{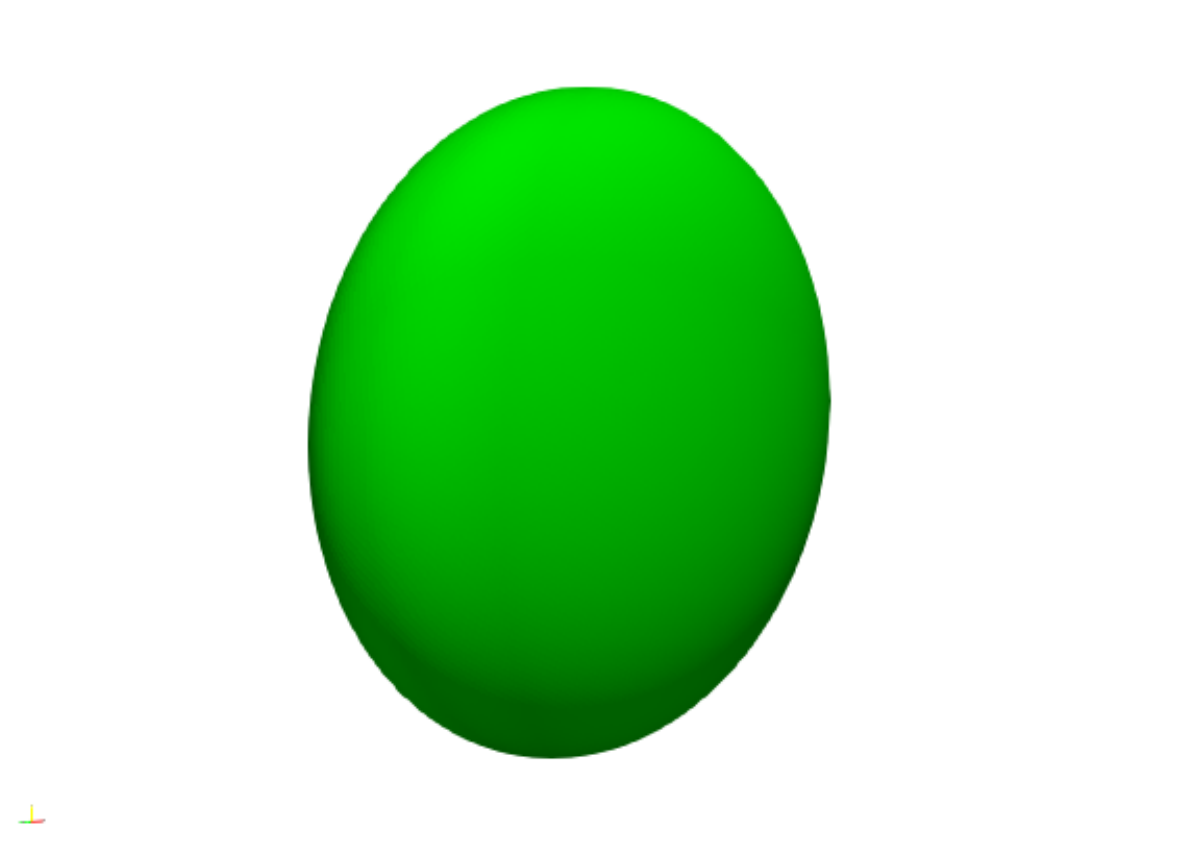}
  \caption{iteration 0}
\end{subfigure}%
\begin{subfigure}{.4\textwidth}
  \centering
  \includegraphics[scale=0.4]{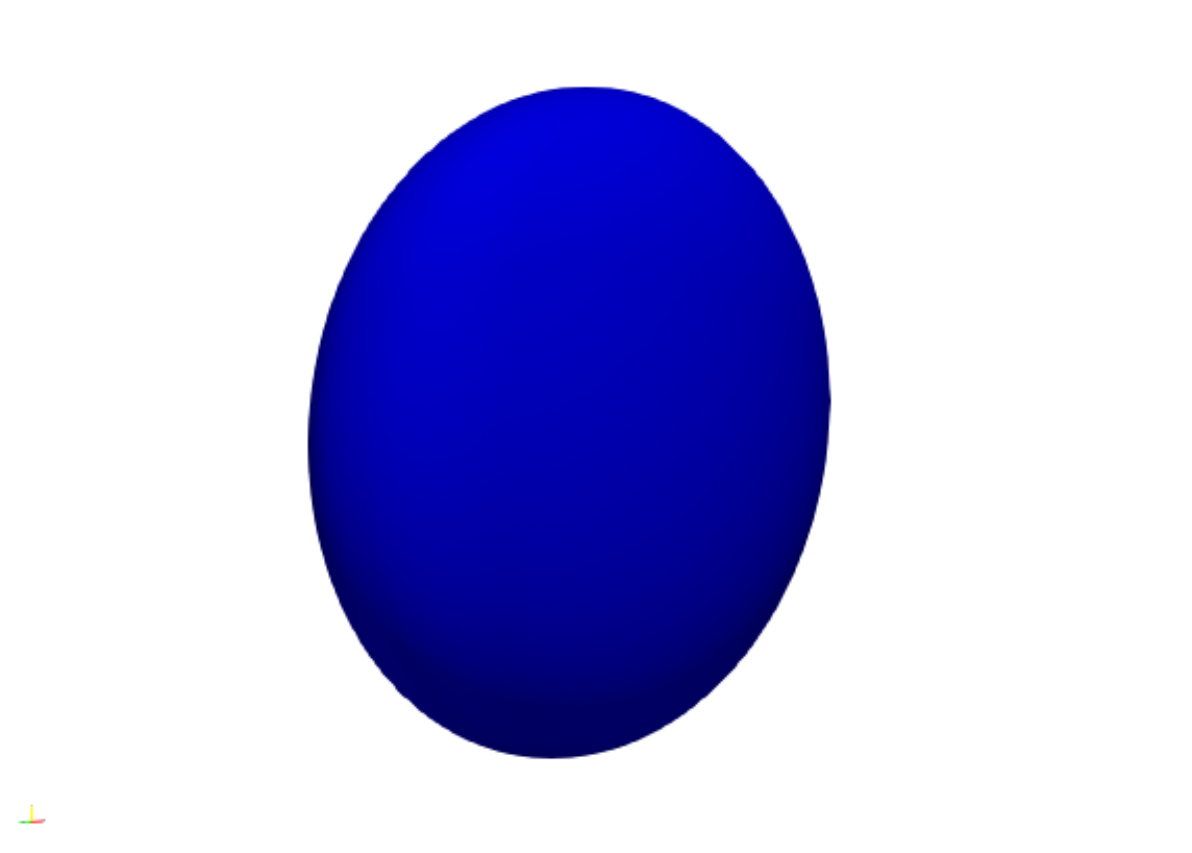}
  \caption{iteration 1}
  \label{fig:two_phase_first}
\end{subfigure}
\begin{subfigure}{.4\textwidth}
  \centering
  \includegraphics[scale=0.4]{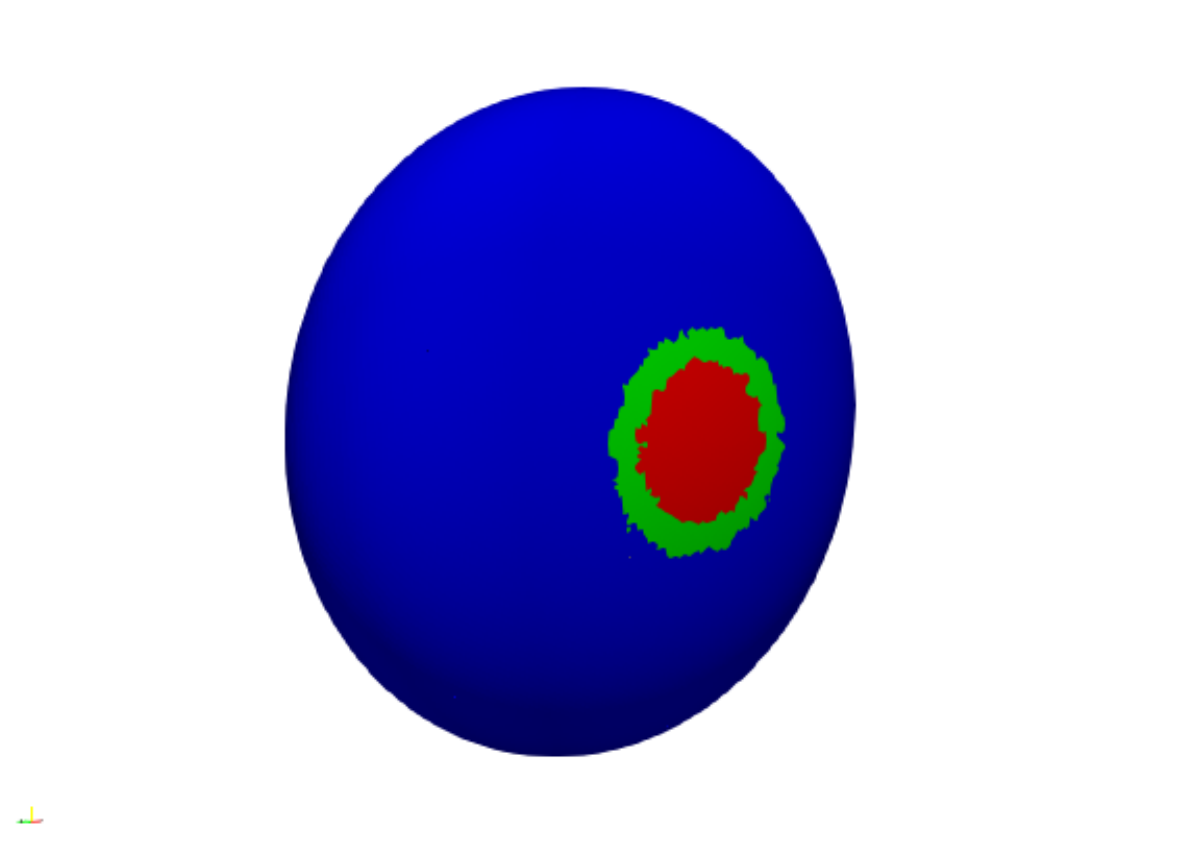}
  \caption{iteration 2}
\end{subfigure}
\begin{subfigure}{.49\textwidth}
  \centering
  \includegraphics[scale=0.4]{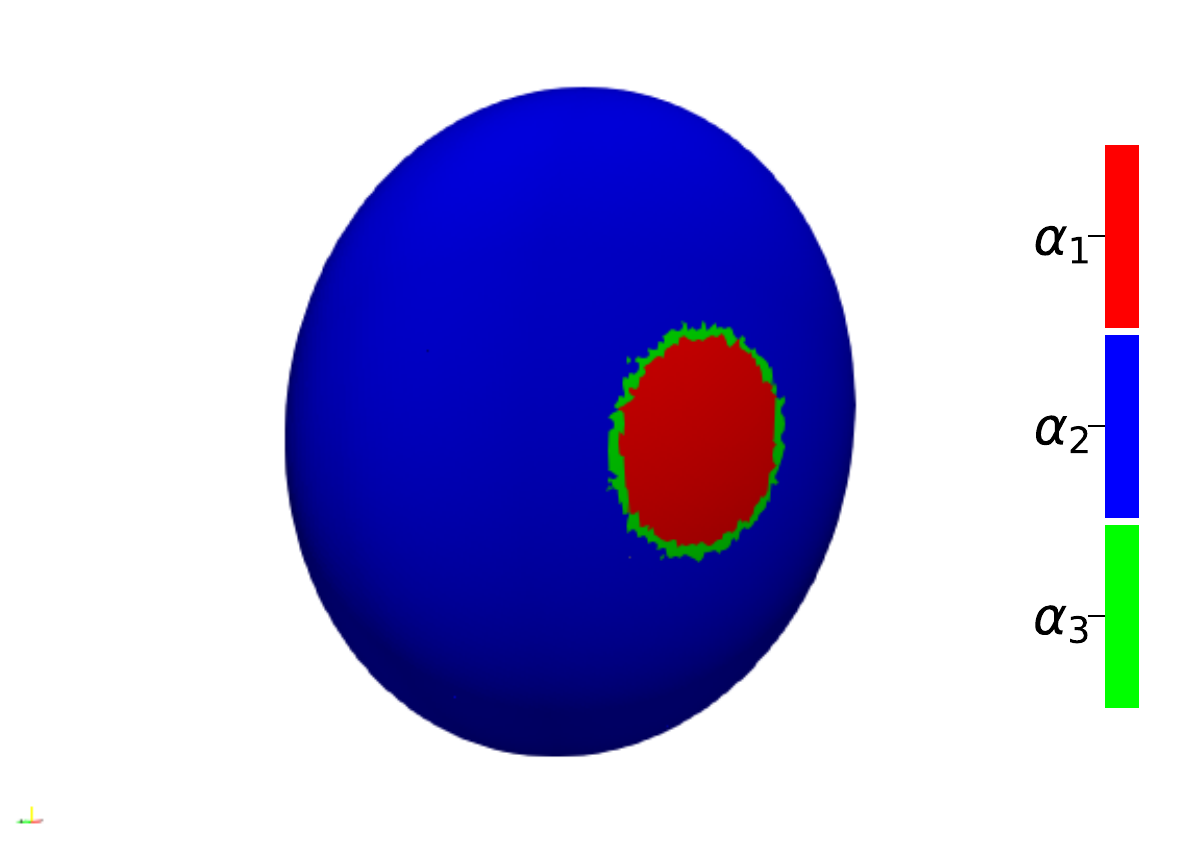}
  \caption{iteration 48}
  \label{fig:two_phase_final}
\end{subfigure}
\caption{Several iterations and the corresponding shape distributions}
\label{fig:two_phases_iterations}
\end{figure}

\begin{figure}
\centering
\begin{subfigure}{.49\textwidth}
  \centering
  \includegraphics[scale=0.4]{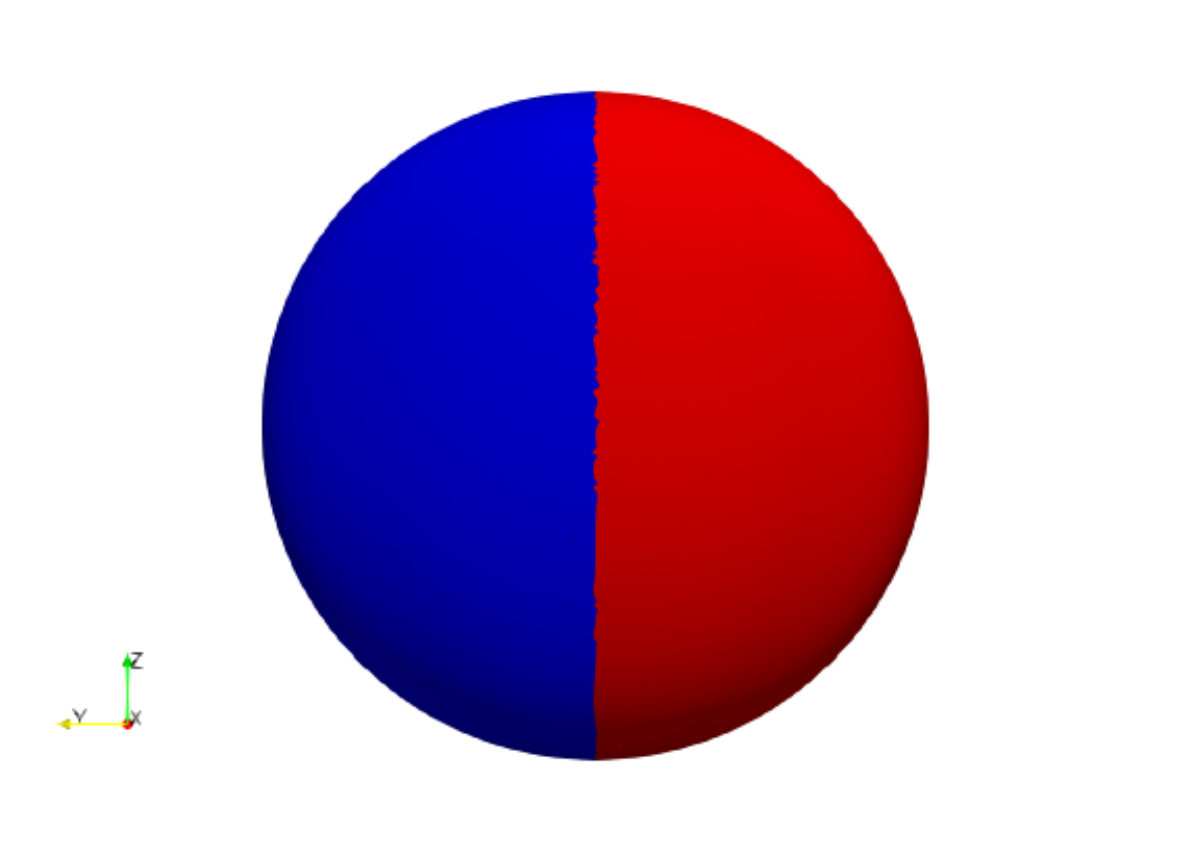}
  \caption{$y-z$ plane ($-x$)}
  \label{fig:three_phase_ref_yz}
\end{subfigure}%
\begin{subfigure}{.49\textwidth}
  \centering
  \includegraphics[scale=0.4]{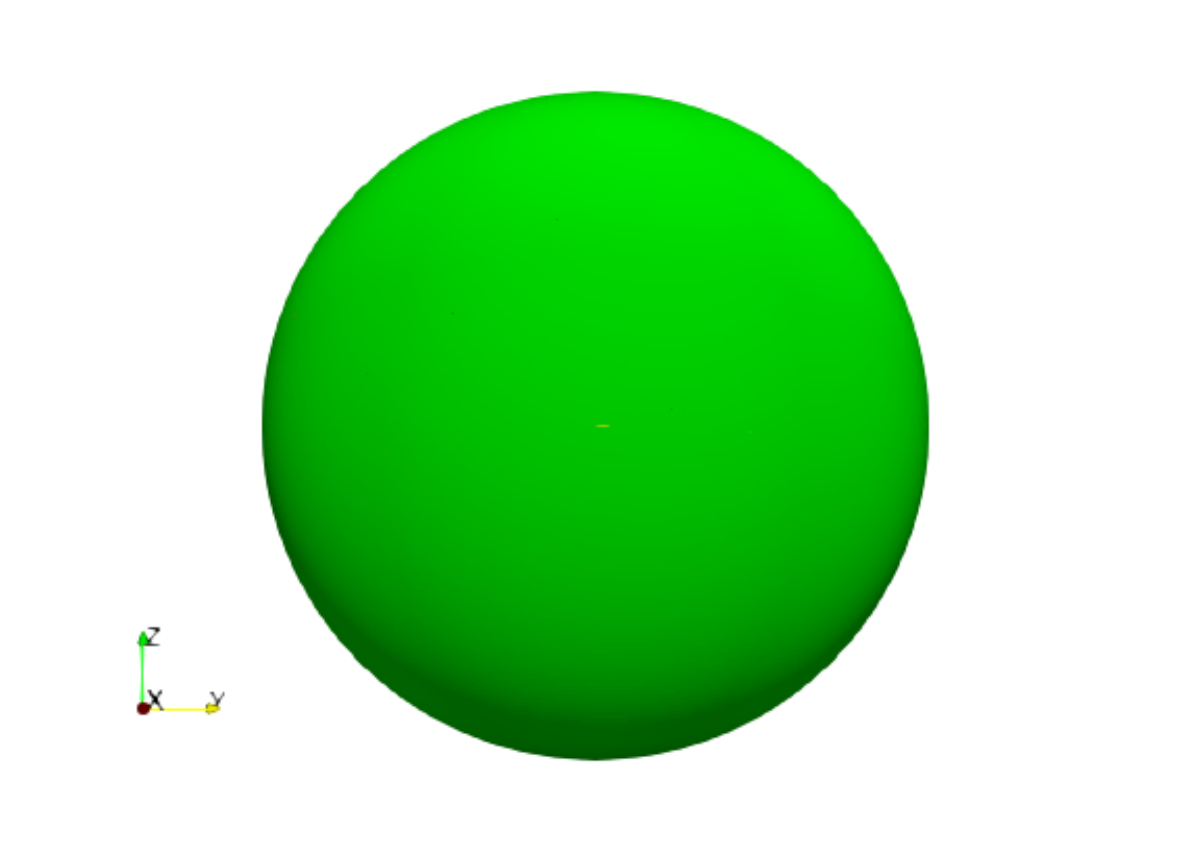}
  \caption{$y-z$ plane ($+x$)}
  \label{fig:three_phase_ref_yz2}
\end{subfigure}
\begin{subfigure}{.4\textwidth}
  \centering
  \includegraphics[scale=0.4]{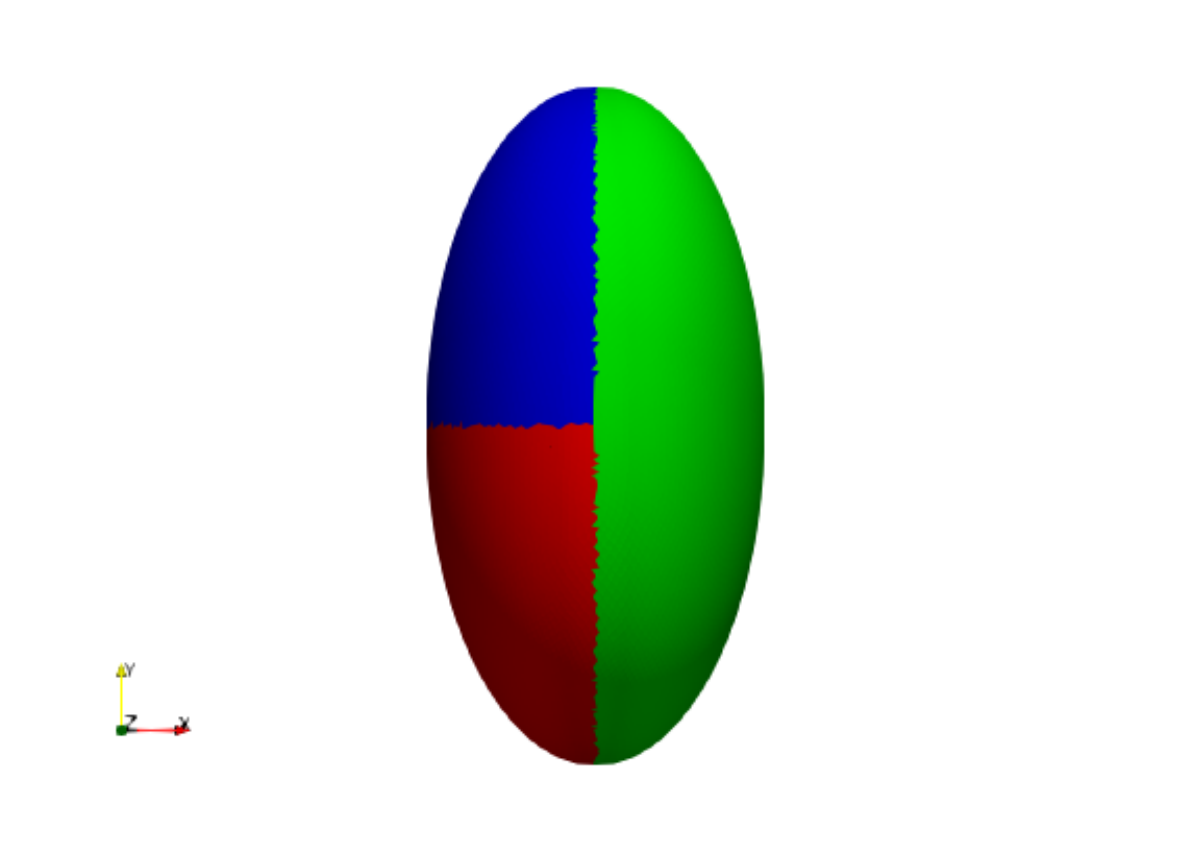}
\caption{$x-y$ plane}
  \label{fig:three_phase_ref_xy}
\end{subfigure}%
\begin{subfigure}{.4\textwidth}
  \centering
  \includegraphics[scale=0.4]{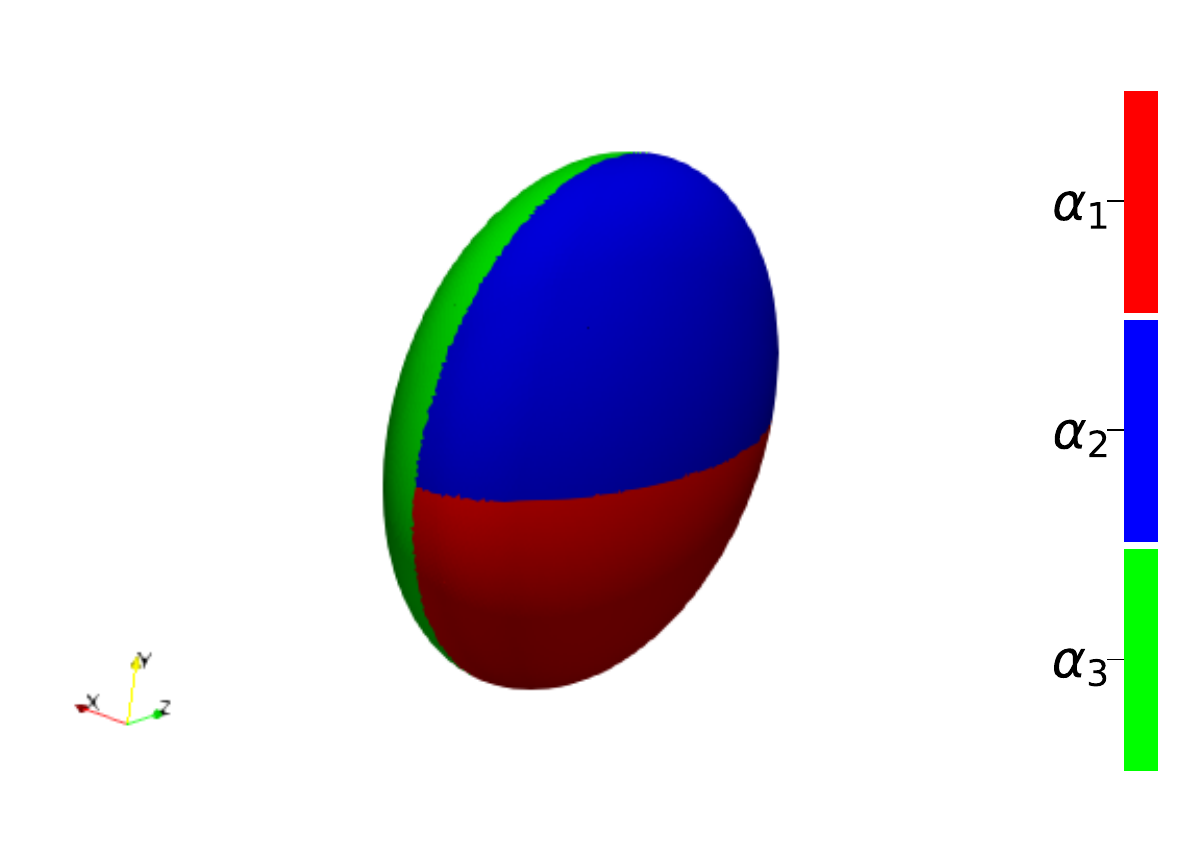}
  \caption{$x-y-z$-plane}
  \label{fig:three_phase_ref_xyz}
\end{subfigure}
\caption{Optimal shape $S_{opt}^b$}
\label{fig:three_phases_reference}
\end{figure}

\begin{figure}
\centering
\begin{subfigure}{.49\textwidth}
  \centering
  \includegraphics[scale=0.4]{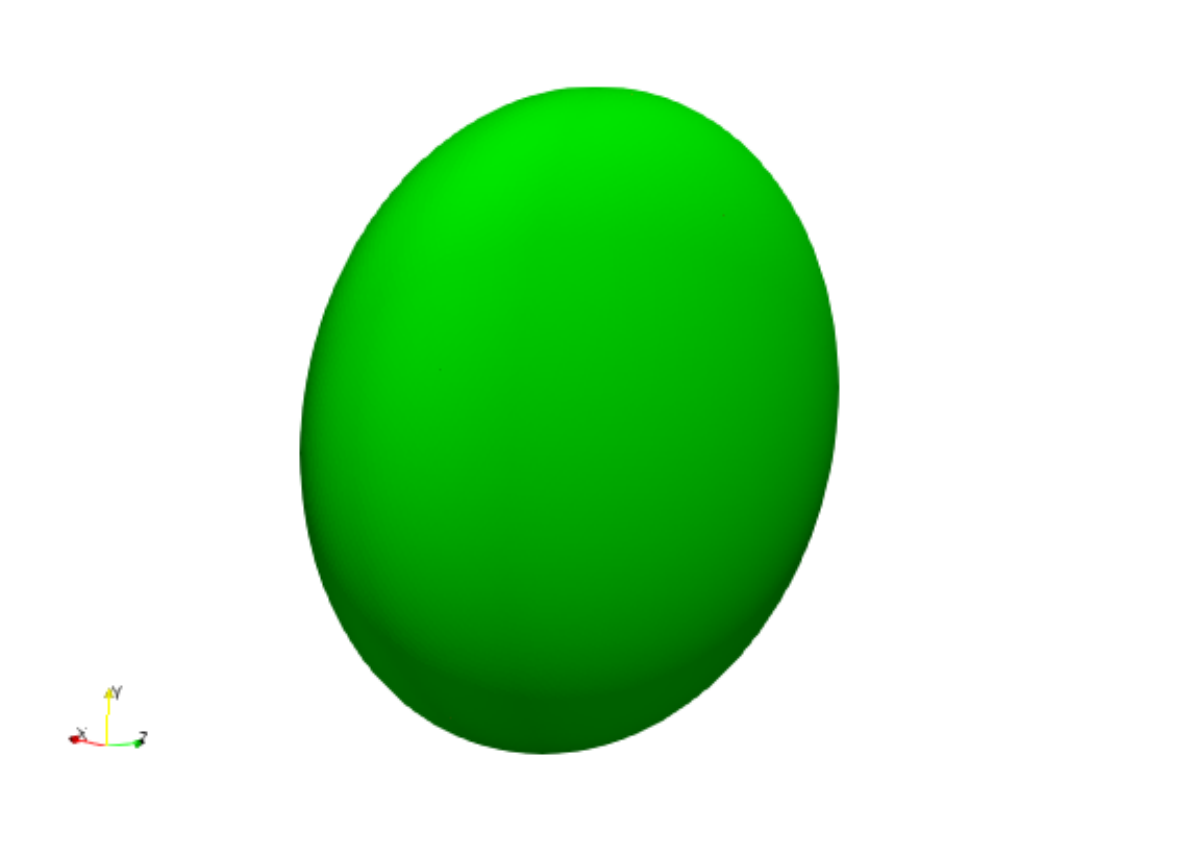}
  \caption{iteration 0}
\end{subfigure}%
\begin{subfigure}{.4\textwidth}
  \centering
  \includegraphics[scale=0.4]{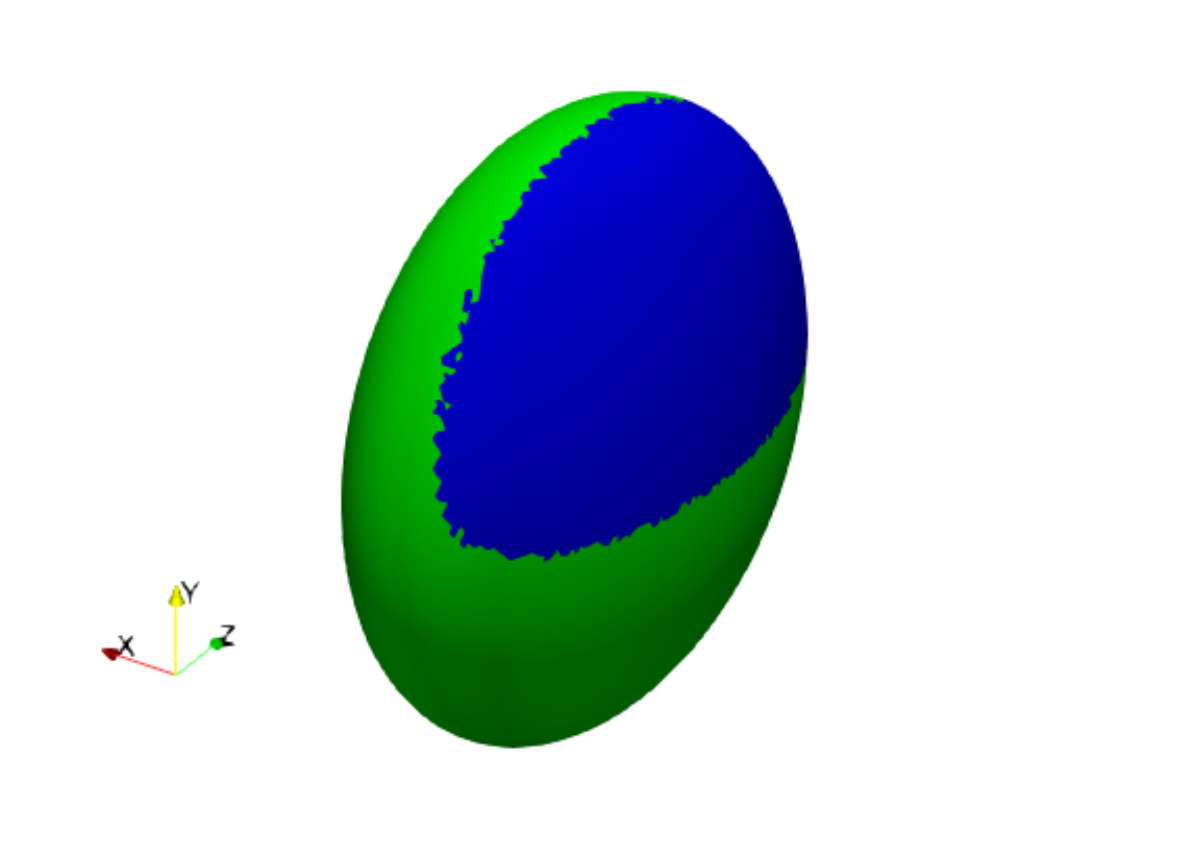}
  \caption{iteration 1}
\end{subfigure}
\begin{subfigure}{.4\textwidth}
  \centering
  \includegraphics[scale=0.4]{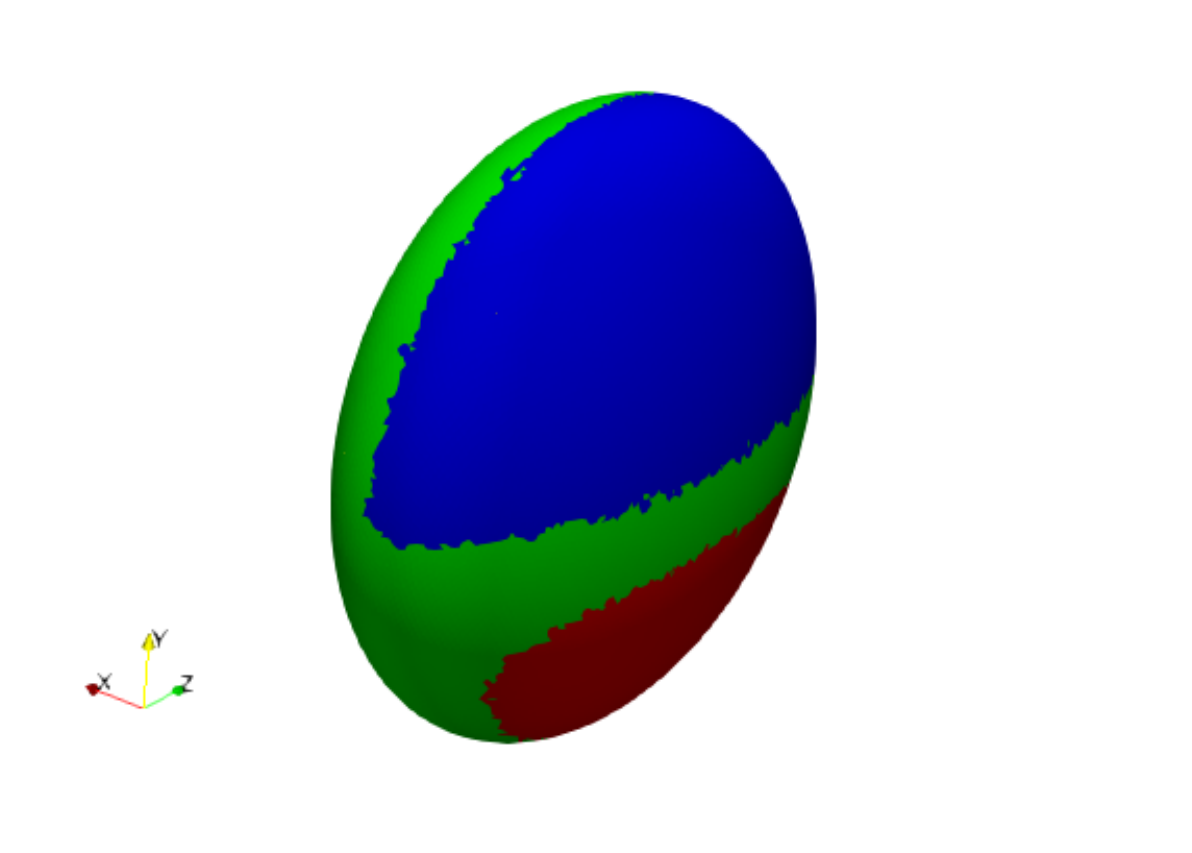}
  \caption{iteration 4}
\end{subfigure}
\begin{subfigure}{.49\textwidth}
  \centering
  \includegraphics[scale=0.4]{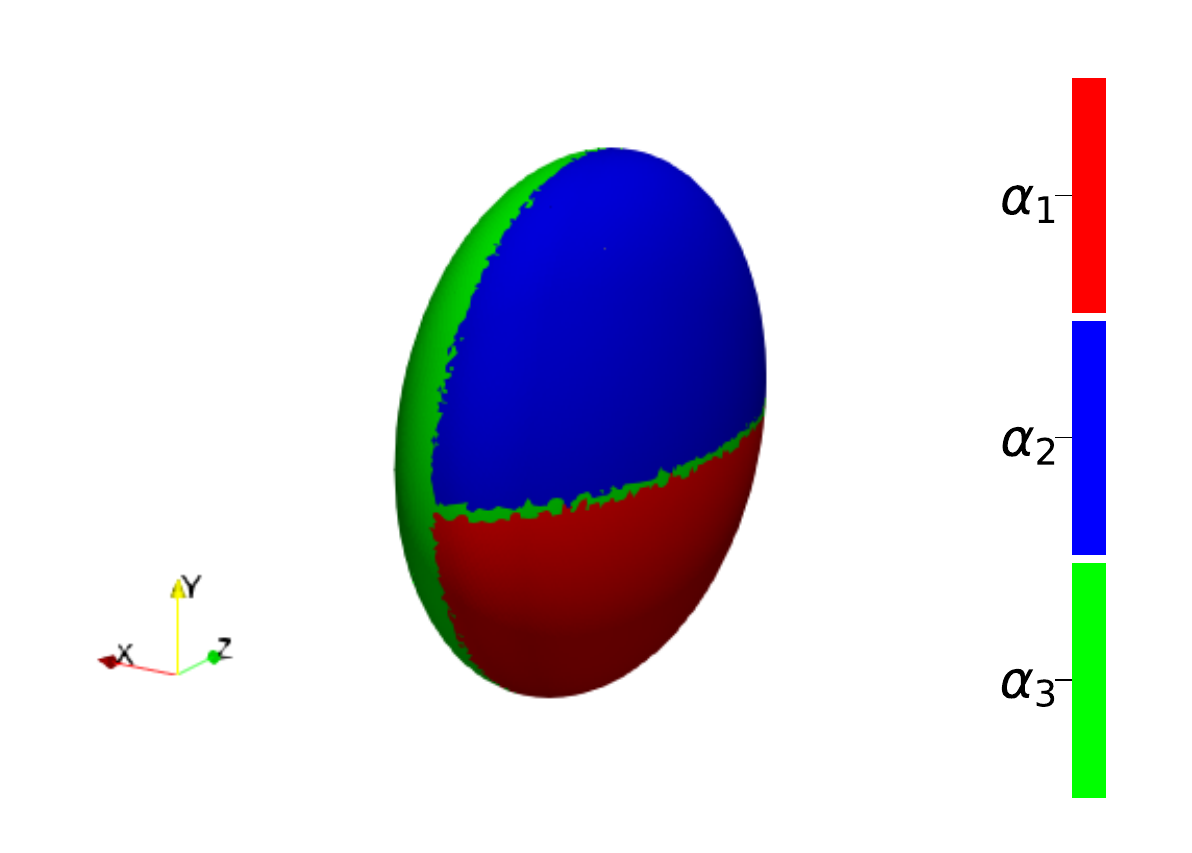}
  \caption{iteration 48}
  \label{fig:three_phase_final}
\end{subfigure}
\caption{Several iterations and the corresponding shape distributions}
\label{fig:three_phases_iterations}
\end{figure}

\section{Conclusion}
In this paper we derived the topological state derivative and the topological derivative of the Poisson equation where the shape variable appears in the inhomogenous Dirichlet boundary condition. We used the sensitivities in a multimaterial level-set algorithm and showed the pertinents of the method. The numerics are limited to three shapes, but it would be interesting to extend the simulations to arbitrary number of shapes and also to include the optimisation of the values of piecewise constant function on the boundary for which we already derived the topological derivative. Also to allow a variable number $M$ of shapes would be interesting to study in a future work as in our approach the number $M$ is fixed.

\bibliographystyle{plain}
\bibliography{dirichlet_control}
\end{document}

%% file: slidesymbols.tex
\providecommand{\limsup}{\operatorname*{limsup}}

\newcommand{\VR}{{\mathbf{R}}}

\providecommand{\Ca}{{\cal A}}
\providecommand{\Cb}{{\cal B}}

\providecommand{\Cj}{{\cal J}}

\providecommand{\Cm}{{\cal M}}